\documentclass[numbook,envcountsame,smallcondensed,draft]{amsart}

\usepackage{amssymb,amsmath,amsfonts,mathrsfs,cite,mathptmx,graphicx,gastex,rotating}

\DeclareMathOperator{\con}{con}
\DeclareMathOperator{\simple}{sim}
\DeclareMathOperator{\mul}{mul}
\DeclareMathOperator{\occ}{occ}
\DeclareMathOperator{\var}{var}
\DeclareMathOperator{\FIC}{FIC}

\newtheorem{theorem}{Theorem}[section]
\newtheorem{proposition}[theorem]{Proposition}
\newtheorem{lemma}[theorem]{Lemma}
\newtheorem{corollary}[theorem]{Corollary}

\theoremstyle{definition}
\newtheorem{remark}{Remark}

\numberwithin{equation}{section}

\makeatletter

\renewcommand*\subjclass[2][2010]{\def\@subjclass{#2}\@ifundefined{subjclassname@#1}{\ClassWarning{\@classname}{Unknown edition (#1) of Mathematics Subject Classification; using '2010'.}}{\@xp\let\@xp\subjclassname\csname subjclassname@#1\endcsname}}

\renewcommand{\subjclassname}{\textup{2010} Mathematics Subject Classification}

\makeatother

\begin{document}

\title[Varieties of monoids whose subvariety lattice is distributive]{Varieties of aperiodic monoids with central idempotents whose subvariety lattice is distributive}
\thanks{The work is supported by the Ministry of Science and Higher Education of the Russian Federation (project FEUZ-2020-0016).}

\author{Sergey V. Gusev}

\address{Ural Federal University, Institute of Natural Sciences and Mathematics, Lenina 51, Ekaterinburg 620000, Russia}

\email{sergey.gusb@gmail.com}

\begin{abstract}
We completely classify all varieties of aperiodic monoids with central idempotents whose subvariety lattice is distributive.
\end{abstract}

\keywords{Monoid, aperiodic monoid, monoid with central idempotents, variety, subvariety lattice, distributive lattice.}

\subjclass{20M07}

\maketitle

\section{Introduction and summary}
\label{Sec: introduction}

A variety $\mathbf V$ is \textit{distributive} if its lattice $\mathfrak L(\mathbf V)$ of subvarieties is distributive.

There was genuine interest to investigate distributive varieties of groups. 
At the turn of the 1960s and 1970s, there were a lot of articles on this topic (see Cossey~\cite{Cossey-69}, Kov\'acs and Newman~\cite{Kovacs-Newman-71} and Roman'kov~\cite{Romankov-70}, for instance). 
However, over time, the activity began to fade. 
The general problem of describing distributive varieties of groups turned out to be highly infeasible.
Here it suffices to refer to the result of Kozhevnikov~\cite{Kozhevnikov-12}, which implies that there exist uncountably many group varieties whose subvariety lattice is isomorphic to the 3-element chain.

The problem of describing (in term of identities) distributive varieties of rings was raised by Bokut' in 1976 in~\cite[Problem~19]{Dnestrovskaja-tetrad}. 
This problem remains open so far.
Quoting from a paper~\cite{Volkov-10} by Volkov on this topic, `There is extensive literature on the subject so that even the mere list of relevant publications is far too long to be placed here. 
Roughly speaking, one may characterize the current stage of investigations as a period of searching for a border separating varieties with distributive and non-distributive subvariety lattice'.

In 1979, Shevrin~\cite[Problem~2.60a]{sverdlovsk-tetrad} posed the problem of classifying all distributive varieties of semigroups. 
This problem includes the problem of identifying all distributive varieties of periodic groups.
In view of the above-mentioned result by Kozhevnikov~\cite{Kozhevnikov-12}, the last problem seems to be extremely difficult.
So, it is natural to speak about the classification modulo group varieties here.
In the early 1990s, in a series of papers, Volkov described distributive varieties of semigroups in a very wide partial case, resulting in an almost complete description modulo group varieties (see Shevrin at al.~\cite[Section~11]{Shevrin-Vernikov-Volkov-09} for more details).

The present article is concerned with the distributive varieties of \textit{monoids}, i.e., semigroups with an identity element.
Even though monoids are very similar to semigroups, the story turns out to be very different and difficult.
Distributive varieties of monoids have not been systematically examined earlier, although non-trivial examples of such varieties have long been known.
We mean the variety of all commutative monoids (Head~\cite{Head-68}) and the variety of all idempotent monoids (Wismath~\cite{Wismath-86}).
Since the middle of the 2000s, papers began to appear with some other non-trivial examples of distributive varieties of monoids (see Gusev~\cite{Gusev-19,Gusev-20}, Gusev and Sapir~\cite{Gusev-Sapir-22}, Gusev and Vernikov~\cite{Gusev-Vernikov-18,Gusev-Vernikov-21}, Jackson~\cite{Jackson-05}, Jackson and Lee~\cite{Jackson-Lee-18}, Lee~\cite{Lee-08,Lee-12b,Lee-14}, Zhang and Luo~\cite{Zhang-Luo-19}).

The present paper is the first attempt at a systematic study of distributive varieties of monoids. 
As in the semigroup case, in view of the result by Kozhevnikov~\cite{Kozhevnikov-12}, the general problem of classifying distributive varieties of monoids seems to be extremely difficult because it includes the problem of identifying all distributive varieties of periodic groups.
Thus, it is natural to begin with the study of monoid varieties with the mentioned property within the class of monoids that do not contain non-trivial subgroups.
Such monoids are called \textit{aperiodic}.
Nevertheless, experience suggests that even the problem of classifying distributive varieties of aperiodic monoids remains quite difficult.
So, it is natural at first to try solving this problem within some subclass of the class of all aperiodic monoids. 

The class $\mathbf A_\mathsf{cen}$ of aperiodic monoids with central idempotents is a natural candidate.
This class is quite wide. 
It includes, in particular, all \textit{nilpotent} monoids, that is, monoids obtained from nilsemigroups by adjoining a new identity element.
Subvarieties of $\mathbf A_\mathsf{cen}$ have been intensively studied for two last decades. 
This class is rich in examples of varieties interesting from specific points of view (see Gusev~\cite{Gusev-19}, Gusev and Lee~\cite{Gusev-Lee-20}, Jackson~\cite{Jackson-05,Jackson-15}, Jackson and Lee~\cite{Jackson-Lee-18}, Jackson and Zhang~\cite{Jackson-Zhang-21}).
Besides that, one managed to completely describe subvarieties of $\mathbf A_\mathsf{cen}$ with some natural and important properties: hereditary finitely based varieties (Lee~\cite{Lee-12a}), almost Cross varieties (Lee~\cite{Lee-13}) and inherently non-finitely generated varieties (Lee~\cite{Lee-14}).
The main result of the present paper naturally fits into this series of results.
Namely, we completely classify distributive subvarieties of $\mathbf A_\mathsf{cen}$.

To formulate the main result of the article, we need some definitions and notation. 
Let $\mathfrak X$ be a countably infinite set called an \textit{alphabet}. 
As usual, we denote by $\mathfrak X^\ast$ the free monoid over the alphabet $\mathfrak X$; elements of $\mathfrak X^\ast$ are called \textit{words}, while elements of $\mathfrak X$ are said to be \textit{letters}. 
Words unlike letters are written in bold. 
An identity is written as $\mathbf u \approx \mathbf v$, where $\mathbf u,\mathbf v \in \mathfrak X^\ast$; it is \textit{non-trivial} if $\mathbf u \ne \mathbf v$.

As usual, $\mathbb N$ denote the set of all natural numbers. 
For any $n\in\mathbb N$, we denote by $S_n$ the full symmetric group on the set $\{1,2,\dots,n\}$. 
For convenience, we put $S_0=S_1$. 
Let $\mathbb N_0=\mathbb N\cup\{0\}$. 
For any $n,m,k\in\mathbb N_0$, $\rho\in S_{n+m}$ and $\tau\in S_{n+m+k}$, we define the words:
$$
\begin{aligned}
\mathbf a_{n,m}[\rho]&=\biggl(\prod_{i=1}^n z_it_i\biggr)x\biggl(\prod_{i=1}^{n+m} z_{i\rho}\biggr)x\biggl(\prod_{i=n+1}^{n+m} t_iz_i\biggr),\\
\mathbf a_{n,m}^\prime[\rho]&=\biggl(\prod_{i=1}^n z_it_i\biggr)x^2\biggl(\prod_{i=1}^{n+m} z_{i\rho}\biggr)\biggl(\prod_{i=n+1}^{n+m} t_iz_i\biggr),\\
\mathbf c_{n,m,k}[\tau]&=\biggl(\prod_{i=1}^n z_it_i\biggr)xyt\biggl(\prod_{i=n+1}^{n+m} z_it_i\biggr)x\biggl(\prod_{i=1}^{n+m+k} z_{i\tau}\biggr)y\biggl(\prod_{i=n+m+1}^{n+m+k} t_iz_i\biggr),\\
\mathbf c_{n,m,k}^\prime[\tau]&=\biggl(\prod_{i=1}^n z_it_i\biggr)yxt\biggl(\prod_{i=n+1}^{n+m} z_it_i\biggr)x\biggl(\prod_{i=1}^{n+m+k} z_{i\tau}\biggr)y\biggl(\prod_{i=n+m+1}^{n+m+k} t_iz_i\biggr).
\end{aligned}
$$
We denote by $\mathbf d_{n,m,k}[\tau]$ and $\mathbf d_{n,m,k}^\prime[\tau]$ the words obtained from the words $\mathbf c_{n,m,k}[\tau]$ and $\mathbf c_{n,m,k}^\prime[\tau]$, respectively, when reading the last words from right to left. 
We fix notation for the following two identities:
$$
\begin{aligned}
\alpha:\enskip xzytxy\approx xzytyx\ \ \text{ and } \ \ 
\beta:\enskip xzxyty\approx xzyxty.
\end{aligned}
$$
Let $\var\Sigma$ denote the monoid variety given by a set $\Sigma$ of identities. 
We fix notation for the following monoid varieties:
$$
\begin{aligned}
&\mathbf P_n=\var
\left\{
\begin{array}{l}
x^n\approx x^{n+1},\,x^2y\approx yx^2,\,\mathbf a_{k,\ell}[\rho] \approx \mathbf a_{k,\ell}^\prime[\rho],\\
\mathbf c_{k,\ell,m}[\tau]\approx\mathbf c_{k,\ell,m}^\prime[\tau],\,\mathbf d_{k,\ell,m}[\tau]\approx\mathbf d_{k,\ell,m}^\prime[\tau]
\end{array}
\middle\vert
\begin{array}{l}
k,\ell,m\in\mathbb N_0,\\
\rho\in S_{k+\ell},\ \tau\in S_{k+\ell+m}
\end{array}
\right\}
,\\
&\mathbf Q_n=\var\{x^n\approx x^{n+1},\,x^ny\approx yx^n,\,x^2y\approx xyx\},\\
&\mathbf R_n=\var\{x^n\approx x^{n+1},\,x^2y\approx yx^2,\,\alpha,\,\beta\},
\end{aligned}
$$
where $n\in \mathbb N$.
By $\mathbf V^\delta$ we denote the monoid variety \textit{dual} to the variety $\mathbf V$ (in other words, $\mathbf V^\delta$ consists of monoids dual to members of $\mathbf V$).

Our main result is the following

\begin{theorem}
\label{T: A_cen}
A subvariety of $\mathbf A_\mathsf{cen}$ is distributive if and only if it is contained in one of the varieties $\mathbf P_n$, $\mathbf Q_n$, $\mathbf Q_n^\delta$, $\mathbf R_n$ or $\mathbf R_n^\delta$ for some $n \in \mathbb N$.
\end{theorem}

Notice that the proof Theorem~\ref{T: A_cen} implies that set of all distributive subvarieties of $\mathbf A_\mathsf{cen}$ is countably infinite (see Remark~\ref{R: countably infinite} at the end of Section~\ref{Sec: proof}).

The article consists of five sections. 
Some background results are first given in Section~\ref{Sec: preliminaries}. 
Section~\ref{Sec: non-distributive} contains several examples of non-distributive varieties of monoids.
In Section~\ref{Sec: auxiliary results}, we prove a number of auxiliary assertions.
Results from Sections~\ref{Sec: non-distributive} and~\ref{Sec: auxiliary results} will then be used in Section~\ref{Sec: proof} to prove Theorem~\ref{T: A_cen}.

\section{Preliminaries}
\label{Sec: preliminaries}

Acquaintance with rudiments of universal algebra is assumed of the reader.
Refer to the monograph of Burris and Sankappanavar~\cite{Burris-Sankappanavar-81} for more information.

\subsection{Deduction}

An identity $\mathbf u \approx \mathbf v$ is \textit{directly deducible} from an identity $\mathbf s \approx \mathbf t$ if there exist some words $\mathbf a,\mathbf b \in \mathfrak X^\ast$ and substitution $\phi\colon \mathfrak X \to \mathfrak X^\ast$ such that $\{ \mathbf u, \mathbf v \} = \{ \mathbf a\phi(\mathbf s)\mathbf b,\mathbf a\phi(\mathbf t)\mathbf b \}$.
A non-trivial identity $\mathbf u \approx \mathbf v$ is \textit{deducible} from a set $\Sigma$ of identities if there exists some finite sequence $\mathbf u = \mathbf w_0, \mathbf w_1, \ldots, \mathbf w_m = \mathbf v$ of words such that each identity $\mathbf w_i \approx \mathbf w_{i+1}$ is directly deducible from some identity in $\Sigma$.

The following assertion is a specialization for monoids of a well-known universal-algebraic fact (see Burris and Sankappanavar~\cite[Theorem~II.14.19]{Burris-Sankappanavar-81}, for instance).

\begin{proposition}
\label{P: deduction}
Let $\mathbf V$ be the variety defined by some set $\Sigma$ of identities.
Then $\mathbf V$ satisfies an identity $\mathbf u \approx \mathbf v$ if and only if $\mathbf u \approx \mathbf v$ is deducible from $\Sigma$.\qed
\end{proposition}

\subsection{Factor monoids}
\label{Subsec: factor monoids}

For any set of words $W$, the \textit{factor monoid of} $W$, denoted by $M(W)$, is the monoid that consists of all factors of the words of $W$ and a zero element $0$, with multiplication~$\cdot$ given by 
$$ 
\mathbf u \cdot \mathbf v = 
\begin{cases} 
\mathbf u\mathbf v & \text{if $\mathbf u\mathbf v$ is a factor of some word in $W$}, \\ 
0 & \text{otherwise}; 
\end{cases} 
$$
the empty word, more conveniently written as~$1$, is the identity element of $M(W)$.
This construction was used by Perkins~\cite{Perkins-69} for exhibiting the first example of a non-finitely based finite semigroup.
Since the beginning of the millennium, such monoids were consistently and systematically studied in the many articles (see Jackson~\cite{Jackson-05}, Jackson and Sapir~\cite{Jackson-Sapir-00}, Sapir~\cite{Sapir-19}, for instance).

A word $\mathbf w$ is an \textit{isoterm} for a variety $\mathbf V$ if $\mathbf V$ violates any non-trivial identity of the form $\mathbf w \approx \mathbf w^\prime$.
Equivalently, $\mathbf w$ is an isoterm for $\mathbf V$ if and only if the identities satisfied by $\mathbf V$ cannot be used to convert $\mathbf w$ into a different word.

Given any word $\mathbf w$, let $\mathbf M(W)$ denote the variety generated by the factor monoid $M(W)$.
One advantage in working with factor monoids is the relative ease of checking if a variety $\mathbf M(W)$ is contained in some given variety.
 
\begin{lemma}[{Jackson~\cite[Lemma~3.3]{Jackson-05}}]
\label{L: M(W) in V}
Let $\mathbf V$ be a monoid variety and $W$ be a set of words. 
Then $M(W)$ lies in $\mathbf V$ if and only if each word in $W$ is an isoterm for $\mathbf V$.\qed
\end{lemma}

For brevity, if $\mathbf w_1,\mathbf w_2,\dots,\mathbf w_k$ are words, then we write $M(\mathbf w_1,\mathbf w_2,\dots,\mathbf w_k)$ [respectively, $\mathbf M(\mathbf w_1,\mathbf w_2,\dots,\mathbf w_k)$] rather than $M(\{\mathbf w_1,\mathbf w_2,\dots,\mathbf w_k\})$ [respectively, $\mathbf M(\{\mathbf w_1,\mathbf w_2,\dots,\mathbf w_k\})$].

\subsection{Subvarieties of $\mathbf A_\mathsf{cen}$}

The following fact is well known and can be easily verified.

\begin{lemma}
\label{L: subvariety of A_cen}
Any subvariety of $\mathbf A_\mathsf{cen}$ satisfies the identities
\begin{equation}
\label{x^n=x^{n+1} and x^ny=yx^n}
x^n \approx x^{n+1}\ \text{ and } \ x^ny \approx yx^n
\end{equation}
for some $n\in \mathbb N$.\qed
\end{lemma}

Actually, the following fact is well known. 
We provide its proof for the sake of completeness.

\begin{lemma}
\label{L: does not contain M(xy)}
Let $\mathbf V$ be a subvariety of $\mathbf A_\mathsf{cen}$. 
If $M(xy)\notin\mathbf V$, then $\mathbf V$ is commutative.
\end{lemma}

\begin{proof}
If $x$ is an isoterm for $\mathbf V$, then $\mathbf V$ satisfies $xy\approx yx$ by Lemma~\ref{L: M(W) in V}.
If $x$ is not an isoterm for $\mathbf V$, then $\mathbf V$ is \textit{completely regular}, i.e., $\mathbf V$ consists of unions of groups by~\cite[Lemma~2.4 and Corollary~2.6]{Gusev-Vernikov-18}. 
It is well known that every completely regular variety of aperiodic monoids consists of idempotent monoids.
In view of Lemma~\ref{L: subvariety of A_cen}, $\mathbf V$ satisfies the identity $x^ny\approx yx^n$ for some $n\ge1$.
It remains to notice that every idempotent monoid satisfying this identity is commutative.
\end{proof}

\begin{corollary}
\label{C: xy=yx in X wedge Y}
Let $\mathbf X$ and $\mathbf Y$ be subvarieties of $\mathbf A_\mathsf{cen}$. 
If $\mathbf X\wedge\mathbf Y$ is commutative, then either $\mathbf X$ or $\mathbf Y$ is commutative.
\end{corollary}

\begin{proof}
If both $\mathbf X$ and $\mathbf Y$ are non-commutative, then $M(xy)\in\mathbf X\wedge\mathbf Y$ by Lemma~\ref{L: does not contain M(xy)}.
To complete the proof, it remains to notice that $M(xy)$ and so $\mathbf X\wedge\mathbf Y$ are non-commutative.
\end{proof}

\subsection{Some known results}
The following fact is obvious.

\begin{lemma}
\label{L: x^n is an isoterm}
Let $\mathbf V$ be a monoid variety and $n\in\mathbb N$. 
Then $x^n$ is not an isoterm for $\mathbf V$ if and only if $\mathbf V$ satisfies the identity $x^n\approx x^m$ for some $m>n$.\qed
\end{lemma}

\begin{lemma}[{Gusev and Vernikov~\cite[Lemma~2.10]{Gusev-Vernikov-21}}]
\label{L: x^n=x^m in X wedge Y}
Let $\mathbf X$ and $\mathbf Y$ be aperiodic monoid varieties. 
If $\mathbf X\wedge\mathbf Y$ satisfies a the identity $x^n\approx x^m$ for some $n,m\in\mathbb N$, then this identity holds in either $\mathbf X$ or $\mathbf Y$.\qed
\end{lemma}

\begin{proposition}[{Head~\cite{Head-68}}]
\label{P: commutative}
Each commutative monoid variety can be defined by the identities $xy\approx yx$ and $x^n\approx x^m$ for some $n,m\in\mathbb N$.\qed
\end{proposition}

The \textit{content} of a word $\mathbf w$, that is, the set of all letters occurring in $\mathbf w$ is denoted by $\con(\mathbf w)$.
For a word $\mathbf w$ and a letter $x$, let $\occ_x(\mathbf w)$ denote the number of occurrences of $x$ in $\mathbf w$. 
A letter $x$ is called \textit{simple} [\textit{multiple}] \textit{in a word} $\mathbf w$ if $\occ_x(\mathbf w)=1$ [respectively, $\occ_x(\mathbf w)>1$]. 
The set of all simple [multiple] letters of a word $\mathbf w$ is denoted by $\simple(\mathbf w)$ [respectively, $\mul(\mathbf w)$]. 
Let $\mathbf w$ be a word and $\simple(\mathbf w)=\{t_1,t_2,\dots,t_m\}$. 
We will assume without loss of generality that $\mathbf w(t_1,t_2,\dots,t_m)=t_1t_2\cdots t_m$. 
Then $\mathbf w=\mathbf w_0t_1\mathbf w_1\cdots t_m\mathbf w_m$ for some words $\mathbf w_0,\mathbf w_1,\dots,\mathbf w_m$. 
The words $\mathbf w_0$, $\mathbf w_1$, \dots, $\mathbf w_m$ are called \textit{blocks} of the word $\mathbf w$. 
The representation of the word $\mathbf w$ as a product of alternating simple in $\mathbf w$ letters and blocks is called a \textit{decomposition} of the word $\mathbf w$.

\begin{lemma}[{Gusev and Vernikov~\cite[Lemma~2.17]{Gusev-Vernikov-21}}]
\label{L: decompositions of u and v}
Let $\mathbf u\approx\mathbf v$ be an identity of $\mathbf M(xy)$. 
Suppose that $\mathbf u_0\prod_{i=1}^m(t_i\mathbf u_i)$ is the decomposition of the word $\mathbf u$. 
Then $\con(\mathbf u)=\con(\mathbf v)$ and the decomposition of the word $\mathbf v$ has the form $\mathbf v_0\prod_{i=1}^m(t_i\mathbf v_i)$.\qed
\end{lemma}

A non-empty word $\mathbf w$ is called \textit{linear} if $\occ_x(\mathbf w)\le 1$ for each letter $x$. 
Let $\mathbf u$ and $\mathbf v$ be words and $\mathbf u_0\prod_{i=1}^m(t_i\mathbf u_i)$ and $\mathbf v_0\prod_{i=1}^m(t_i\mathbf v_i)$ be decompositions of $\mathbf u$ and $\mathbf v$, respectively. 
A letter $x$ is called \textit{linear-balanced in the identity} $\mathbf u\approx\mathbf v$ if $x$ is multiple in $\mathbf u$ and $\occ_x(\mathbf u_i)=\occ_x(\mathbf v_i)\le 1$ for all $i=0,1,\dots,m$; the identity $\mathbf u\approx\mathbf v$ is called \textit{linear-balanced} if any letter $x\in\mul(\mathbf u)\cup\mul(\mathbf v)$ is linear-balanced in this identity. 

\begin{lemma}[{Gusev and Vernikov~\cite[Lemma~3.1]{Gusev-Vernikov-21}}]
\label{L: xt_1x...t_kx is isoterm}
Let $\mathbf V$ be a monoid variety such that the word $\bigl(\prod_{i=1}^k xt_i\bigr)x$ is an isoterm for $\mathbf V$, $\mathbf u$ be a word such that all its blocks are linear words and $\occ_x(\mathbf u)\le k+1$ for every letter $x$. 
Then every identity of the form $\mathbf u\approx\mathbf v$ that holds in the variety $\mathbf V$ is linear-balanced.\qed
\end{lemma}

\begin{lemma}[{Gusev and Vernikov~\cite[Lemma~4.9]{Gusev-Vernikov-18}}]
\label{L: V does not contain M(xyzxty) or M(xtyzxy) or M(xzxyty)}
Let $\mathbf V$ be a monoid variety with $M(xyx)\in\mathbf V$.
\begin{itemize}
\item[\textup{(i)}] If $M(xtyzxy)\notin\mathbf V$, then $\mathbf V$ satisfies the identity $\alpha$.
\item[\textup{(ii)}] If $M(xzxyty)\notin\mathbf V$, then $\mathbf V$ satisfies the identity $\beta$.\qed
\end{itemize}
\end{lemma}

The subvariety of a variety $\mathbf V$ defined by a set $\Sigma$ of identities is denoted by $\mathbf V \Sigma$.

\begin{lemma}[{Gusev and Vernikov~\cite[Lemma~2.19]{Gusev-Vernikov-21}}]
\label{L: smth imply distributivity}
Let $\mathbf V$ be a variety. 
Suppose that there is a set $\Sigma$ of identities such that:
\begin{itemize}
\item[\textup{(i)}] if $\mathbf U\subseteq\mathbf V$, then $\mathbf U=\mathbf V\Phi$ for some subset $\Phi$ of $\Sigma$;
\item[\textup{(ii)}] if $\mathbf U,\mathbf U^\prime\subseteq\mathbf V$ and $\mathbf U\wedge\mathbf U^\prime$ satisfies an identity $\sigma\in\Sigma$, then $\sigma$ holds in either $\mathbf U$ or $\mathbf U^\prime$.
\end{itemize}
Then the the lattice $\mathfrak L(\mathbf V)$ is distributive.\qed
\end{lemma}

\section{Certain varieties with non-distributive subvariety lattice}
\label{Sec: non-distributive}

\subsection{The variety $\mathbf M(x^2y,yx^2)$}
\label{Subsec: M(x^2y,yx^2)}

Let $\mathbf T$ and $\mathbf{SL}$ denote the variety of trivial monoids and the variety of all semilattice monoids, respectively.

\begin{proposition}[{Gusev and Lee~\cite[Proposition~3.1]{Gusev-Lee-21}}]
\label{P: M(x^2y,yx^2)}
The lattice $\mathfrak L(\mathbf M(x^2y,yx^2))$ is given in Fig.~\textup{\ref{L(M(x^2y,yx^2))}}.
In particular, this lattice is modular but not distributive.\qed
\end{proposition}

\begin{figure}[htb]
\unitlength=1mm
\linethickness{0.4pt}
\begin{center}
\begin{picture}(44,69)
\put(26,5){\circle*{1.33}}
\put(26,15){\circle*{1.33}}
\put(26,25){\circle*{1.33}}
\put(16,35){\circle*{1.33}}
\put(36,35){\circle*{1.33}}
\put(6,45){\circle*{1.33}}
\put(26,45){\circle*{1.33}}
\put(16,55){\circle*{1.33}}
\put(36,55){\circle*{1.33}}
\put(26,55){\circle*{1.33}}
\put(26,65){\circle*{1.33}}

\put(36,50){\vector(-2,1){9}}

\put(26,5){\line(0,1){20}}
\put(26,25){\line(1,1){10}}
\put(26,25){\line(-1,1){20}}
\put(16,35){\line(1,1){20}}
\put(36,35){\line(-1,1){20}}
\put(6,45){\line(1,1){20}}
\put(26,45){\line(0,1){20}}
\put(36,55){\line(-1,1){10}}

\put(27,24){\makebox(0,0)[lc]{$\mathbf M(x)$}}
\put(37,35){\makebox(0,0)[lc]{$\mathbf M(x^2)$}}
\put(15,34){\makebox(0,0)[rc]{$\mathbf M(xy)$}}
\put(5,44){\makebox(0,0)[rc]{$\mathbf M(xyx)$}}
\put(36.5,49.5){\makebox(0,0)[lc]{$\mathbf M(x^2y)$}}
\put(26,68){\makebox(0,0)[cc]{$\mathbf M(x^2y,yx^2)$}}
\put(37,56){\makebox(0,0)[lc]{$\mathbf M(yx^2)$}}
\put(27,15){\makebox(0,0)[lc]{$\mathbf{SL}$}}
\put(26,2){\makebox(0,0)[cc]{$\mathbf T$}}
\end{picture}
\end{center}
\caption{The lattice $\mathfrak L(\mathbf M(x^2y,yx^2))$}
\label{L(M(x^2y,yx^2))}
\end{figure}

\subsection{The variety $\mathbf M(xzxyty)\vee \mathbf N$}
\label{Subsec: M(xzxyty) vee N}

Let
$$
\mathbf N=\var\{x^2\approx x^3,\,x^2y\approx yx^2,\,xyxzx\approx x^2yz,\,\alpha,\,\beta\}.
$$
If $\mathbf w$ is a word and $X\subseteq\con(\mathbf w)$, then we denote by $\mathbf w(X)$ the word obtained from $\mathbf w$ by deleting all letters except letters from $X$. 
If $X=\{x_1,x_2,\dots,x_k\}$, then we write $\mathbf w(x_1,x_2,\dots,x_k)$ rather than $\mathbf w(\{x_1,x_2,\dots,x_k\})$. 

\begin{proposition}
\label{P: M(xzxyty) vee N}
The lattice $\mathfrak L(\mathbf M(xzxyty)\vee \mathbf N)$ is not modular.
\end{proposition}

\begin{proof}
Suppose that $\mathbf M(xzxyty)\vee \mathbf N$ satisfies an identity $xysxzytz\approx \mathbf v$ for some $\mathbf v\in\mathfrak X^\ast$.
In view of Lemma~\ref{L: M(W) in V}, $\mathbf v(x,z,s,t)=xsxztz$ and $\mathbf v(y,z,s,t)=yszytz$.
Then $\mathbf v = \mathbf v^\prime sxzytz$, where $\mathbf v^\prime\in\{xy,yx\}$.
Since $\mathbf N$ violates 
\begin{equation}
\label{xyzxy=yxzxy}
xyzxy\approx yxzxy,
\end{equation} 
we have $\mathbf v^\prime\ne yx$.
Therefore, $\mathbf v = xysxzytz$.
We see that $xysxzytz$ is an isoterm for the variety $\mathbf M(xzxyty)\vee \mathbf N$.
Let $\mathbf V=\mathbf M(xysxzytz)$.
Then 
$$
(\mathbf M(xzxyty)\vee \mathbf N)\wedge \mathbf V = \mathbf V
$$
by Lemma~\ref{L: M(W) in V}.
It is routine to check that~\eqref{xyzxy=yxzxy} holds in $M(xysxzytz)$.
Then $\mathbf N\wedge\mathbf V$ satisfies
$$
xysxzytz\stackrel{\beta}\approx xysxyztz\stackrel{\eqref{xyzxy=yxzxy}}\approx yxsxyztz\stackrel{\beta}\approx yxsxzytz.
$$
Clearly, $xysxzytz\approx yxsxzytz$ is satisfied by $\mathbf M(xzxyty)$ as well.
Since the submonoid of $M(xysxzytz)$ generated by $\{x,z,ys,yt,1\}$ is isomorphic to $M(xzxyty)$ and so $\mathbf M(xzxyty)\subset\mathbf V$, we have
$$
\mathbf M(xzxyty)\vee(\mathbf N\wedge\mathbf V)\subset (\mathbf M(xzxyty)\vee \mathbf N)\wedge \mathbf V = \mathbf V.
$$
It follows that the lattice $\mathfrak L(\mathbf M(xzxyty)\vee \mathbf N)$ is not modular.
\end{proof}

\subsection{Varieties of the form $\mathbf M(\mathbf a_{n,m}[\rho])$}
\label{Subsec: M(a_{n,m}[rho])}

For any $n,m\in\mathbb N_0$ and $\rho\in S_{n+m}$, we put $\hat{\mathbf a}_{n,m}[\rho]=(\mathbf a_{n,m}[\rho])_x$.
Let 
$$
\hat{\mathbb N}_0^2 =\{(n,m)\in \mathbb N_0\times \mathbb N_0\mid |n-m|\le 1\}.
$$
For $n,m\in\mathbb N_0$, a permutation $\rho$ from $S_{n+m}$ is a $(n,m)$-\textit{permutation} if, for all $i=1,2,\dots,n+m-1$, one of the following holds:
\begin{itemize}
\item $1\le i\rho\le n$ and $n<(i+1)\rho\le n+m$;
\item $1\le (i+1)\rho\le n$ and $n<i\rho\le n+m$.
\end{itemize}
Evidently, if $\rho$ is a $(n,m)$-permutation, then $(n,m)\in \hat{\mathbb N}_0^2$.
The set of all $(n,m)$-permutations is denoted by $S_{n,m}$.
If $\mathbf w$ is a word and $X\subseteq\con(\mathbf w)$, then we denote by $\mathbf w_X$ the word obtained from $\mathbf w$ by deleting all letters from $X$. 
If $X=\{x\}$, then we write $\mathbf w_x$ rather than $\mathbf w_{\{x\}}$.

The proof of the following lemma is quite analogous to the proof of Lemma~4.10 in~\cite{Gusev-Vernikov-18} and so we omit it.

\begin{lemma}
\label{L: isoterms for M(xzxyty)}
For any $(n,m)\in \hat{\mathbb N}_0^2$ and $\rho\in S_{n,m}$, the word $\hat{\mathbf a}_{n,m}[\rho]$ is an isoterm for the variety $\mathbf M(xzxyty)$.\qed
\end{lemma}

For any non-empty word $\mathbf w$ of length $\ell$, $0\le k \le \ell$ and $0\le m\le\ell-k$, let $\mathbf w[k;m]$ denote a factor of $\mathbf w$ of length $m$ directly succeeding the prefix of $\mathbf w$ of length $k$.
The expression $_{i\mathbf w}x$ means the $i$th occurrence of a letter $x$ in a word $\mathbf w$. 
If the $i$th occurrence of $x$ precedes the $j$th occurrence of $y$ in a word $\mathbf w$, then we write $({_{i\mathbf w}x}) < ({_{j\mathbf w}y})$.
For any monoid variety $\mathbf V$, let $\FIC(\mathbf V)$ denote the fully invariant congruence on $\mathfrak X^\ast$ corresponding to $\mathbf V$.

\begin{proposition}
\label{P: L(M(a_{n,m}[rho])) is not distributive}
The lattice $\mathfrak L(\mathbf M(\mathbf a_{n,m}[\rho]))$ is not distributive for any $(n,m)\in \hat{\mathbb N}_0^2$ and $\rho\in S_{n,m}$.
\end{proposition}

\begin{proof}
There are four possibilities:
\begin{itemize}
\item $n=m$ and $1\le 1\rho\le n$;
\item $n=m$ and $n+1\le 1\rho\le 2n$;
\item $n=m+1$ and so $1\le 1\rho\le n$;
\item $n=m-1$ and so $n+1\le 1\rho\le n+m$.
\end{itemize}
We will consider only the first possibility because the other ones are considered quite analogous.
In this case, $1\le i\theta\le n$ and $n+1\le (i+1)\theta\le 2n$ for any $i=1,3,\dots,2n-1$.

For $k\in \mathbb N$ and an arbitrary permutation $\theta\in S_{k,k}$ with $1\le i\theta\le k$ and $k+1\le (i+1)\theta\le 2k$ for any $i=1,3,\dots,2k-1$, we define the permutation $\theta^\prime\in S_{k+2,k+2}$ as follows:
\begin{itemize}
\item $1\theta^\prime = 1\theta+1$ and $2\theta^\prime = 2\theta+3$;
\item $3\theta^\prime = k+2$ and $4\theta^\prime = 2k+4$;
\item $i\theta^\prime = (i-2)\theta+1$ for any $i=5,7,\dots, 2k-1$;
\item $i\theta^\prime = (i-2)\theta+3$ for any $i=6,8,\dots, 2k$;
\item $(2k+1)\theta^\prime = 1$ and $(2k+2)\theta^\prime = k+3$;
\item $(2k+3)\theta^\prime = (2k-1)\theta+1$ and $(2k+4)\theta^\prime = (2k)\theta+3$.
\end{itemize}
Let $\mathbf a_{k+2,k+2}[\theta^\prime]\approx \mathbf a$ be an identity of $\mathbf M(\mathbf a_{k,k}[\theta])$.
It is easy to see that a non-trivial identity of the form $xzxyty \approx \mathbf w$ implies a non-trivial identity of the form $\mathbf a_{k,k}[\theta] \approx \mathbf w^\prime$.
This observation and Lemma~\ref{L: M(W) in V} imply that $xzxyty$ is an isoterm for $\mathbf M(\mathbf a_{k,k}[\theta])$.
Then $\mathbf a_x=\hat{\mathbf a}_{k+2,k+2}[\theta^\prime]$ by Lemma~\ref{L: isoterms for M(xzxyty)}. 
Since $\mathbf a_{k+2,k+2}[\theta^\prime](X)$ coincides (up to renaming of letters) with $\mathbf a_{k,k}[\theta]$ for 
$$
X=\{x\}\cup\{z_i,t_i\mid 2\le i\le n+1\}\cup\{z_i,t_i\mid k+4\le i\le 2k+3\},
$$ 
Lemma~\ref{L: M(W) in V} implies that $\mathbf a(X)=\mathbf a_{k,k}[\theta](X)$.
In particular, $\occ_x(\mathbf a)=2$.
We notice also that the first and the second occurrences of $x$ in $\mathbf a$ lie in the same block because $xzxyty$ and so $xzx$ are isoterms for $\mathbf M(\mathbf a_{k,k}[\theta])$.
Therefore, $({_{1\mathbf a}t_{n+2}}) < ({_{1\mathbf a}x})$ and $({_{2\mathbf a}x}) < ({_{1\mathbf a}t_{n+3}})$.
It follows that $\mathbf a = \mathbf a_{k+2,k+2}[\theta^\prime]$.
We see that the word $\mathbf a_{k+2,k+2}[\theta^\prime]$ is an isoterm for $\mathbf M(\mathbf a_{k,k}[\theta])$.
Then $\mathbf M(\mathbf a_{k+2,k+2}[\theta^\prime])\subseteq\mathbf M(\mathbf a_{k,k}[\theta])$ by Lemma~\ref{L: M(W) in V}.

Since it suffices to verify that some subvariety of $\mathbf M(\mathbf a_{n,n}[\rho])$ has a non-distributive subvariety lattice, we may further assume without any loss that 
$\rho=\xi^\prime$ for some $\xi\in S_{n-2,n-2}$.
In particular,
\begin{equation}
\label{1rho,(2n-1)rho,2rho,(2n)rho}
2\le 1\rho,(2n-1)\rho\le n-1\ \text{ and }\ n+2\le 2\rho,(2n)\rho\le 2n-1.
\end{equation}

For any $0 \le s \le t \le 6n$, put $\mathbf v_{s,t} = \mathbf p\,\mathbf q[0;s]\,x\,\mathbf q[s;t-s]\,x\,\mathbf q[t;6n-t]\,\mathbf r$, where
$$
\begin{aligned}
&\mathbf p = \biggl(\prod_{i=1}^{1\rho} z_i^{\prime\prime} t_i^{\prime\prime}\biggr)\biggl(\prod_{i=1}^{(2n-1)\rho} z_i^{\prime} t_i^{\prime}\biggr)\biggl(\prod_{i=1}^n z_i t_i\biggr)\biggl(\prod_{i=(2n-1)\rho+1}^n z_i^{\prime} t_i^{\prime}\biggr)\biggl(\prod_{i=1\rho+1}^n z_i^{\prime\prime} t_i^{\prime\prime}\biggr),\\
&\mathbf q = z_{1\rho}z_{2\rho}\cdot\biggl(\prod_{i=1}^{2n} z_{i\rho}^{\prime}\biggr)\biggl(\prod_{i=3}^{2n-2} z_{i\rho}\biggr)\biggl(\prod_{i=1}^{2n} z_{i\rho}^{\prime\prime} \biggr)\cdot z_{(2n-1)\rho}z_{(2n)\rho},\\
&\mathbf r = \biggl(\prod_{i=n+1}^{2\rho} t_i^{\prime\prime}z_i^{\prime\prime}\biggr)\biggl(\prod_{i=n+1}^{(2n)\rho} t_i^{\prime}z_i^{\prime}\biggr)\biggl(\prod_{i=n+1}^{2n} t_iz_i\biggr)\biggl(\prod_{i=(2n)\rho+1}^{2n} z_i^{\prime}t_i^{\prime}\biggr)\biggl(\prod_{i=2\rho+1}^{2n} z_i^{\prime\prime}t_i^{\prime\prime} \biggr).
\end{aligned}
$$
Evidently, $\mathbf v_{0,6n}$ coincides (up to renaming of letters) with $\mathbf a_{3n,3n}[\tau]$ for some $\tau\in S_{3n,3n}$.
Then arguments similar to ones from the second paragraph of this proof imply that the word $\mathbf v_{0,6n}$ is an isoterm for $\mathbf M(\mathbf a_{n,n}[\rho])$.

Let $\mathbf v_{1,6n}\approx \mathbf v$ be an identity of $\mathbf M(\mathbf a_{n,n}[\rho])$.
Arguments similar to ones from the second paragraph of this proof imply that $\mathbf v_x=\mathbf p\mathbf q\mathbf r$.
Note that if
$$
X=\{x,z_{(2n-1)\rho},t_{(2n-1)\rho},z_{(2n)\rho},t_{(2n)\rho}\}\cup\{z_{i\rho}^\prime,t_{i\rho}^\prime\mid 1\le i\le 2n-2\},
$$
then $\mathbf v_{1,6n}(X)$ coincides (up to renaming of letters) with $\mathbf a_{n,n}[\rho]$.
Then $\mathbf v(X)=\mathbf v_{1,6n}(X)$ by Lemma~\ref{L: M(W) in V}. 
In particular, $\occ_x(\mathbf v)=2$.
Arguments similar to ones from the second paragraph of this proof imply that $({_{1\mathbf v}t_n^{\prime\prime}}) < ({_{1\mathbf v}x})$ and $({_{2\mathbf v}x}) < ({_{1\mathbf v}t_{n+1}^{\prime\prime}})$.
It follows that $\mathbf v$ is either $\mathbf v_{0,6n}$ or $\mathbf v_{1,6n}$ or $\mathbf v_{2,6n}$.
Since $\mathbf v_{0,6n}$ is an isoterm for $\mathbf M(\mathbf a_{n,n}[\rho])$, the word $\mathbf v$ cannot coincide with $\mathbf v_{0,6n}$.
Therefore, $\mathbf v\in\{\mathbf v_{1,6n},\mathbf v_{2,6n}\}$. 
Finally, it is routine to check that the identity $\mathbf v_{1,6n}\approx\mathbf v_{2,6n}$ holds in $\mathbf M(\mathbf a_{n,n}[\rho])$.
We see that the set $\{\mathbf v_{1,6n},\mathbf v_{2,6n}\}$ forms a $\FIC(\mathbf M(\mathbf a_{n,n}[\rho]))$-class.
By similar arguments we can show that the set $\{\mathbf v_{0,6n-1},\mathbf v_{0,6n-2}\}$ is a $\FIC(\mathbf M(\mathbf a_{n,n}[\rho]))$-class.

Let 
$$
\begin{aligned}
\mathbf X = \mathbf M(\mathbf a_{n,n}[\rho])\wedge\var\{\mathbf v_{0,6n} \approx \mathbf v_{2,6n}\}\  \text { and } \ 
\mathbf Y = \mathbf M(\mathbf a_{n,n}[\rho])\wedge\var\{\mathbf v_{0,6n} \approx \mathbf v_{0,6n-2}\}
\end{aligned}
$$
and $A=\{\mathbf v_{0,6n},\mathbf v_{1,6n},\mathbf v_{2,6n}\}$.
Consider an identity $\mathbf u \approx \mathbf u^\prime$ of $\mathbf X$ with $\mathbf u\in A$.
We are going to show that $\mathbf u^\prime \in A$.
In view of Proposition~\ref{P: deduction}, we may assume without loss of generality that  either $\mathbf u \approx \mathbf u^\prime$ holds in $\mathbf M(\mathbf a_{n,n}[\rho])$ or $\mathbf u \approx \mathbf u^\prime$ is directly deducible from $\mathbf v_{0,6n} \approx \mathbf v_{2,6n}$. 
In view of the above, $A$ is a union of two $\FIC(\mathbf M(\mathbf a_{n,n}[\rho]))$-classes.
Therefore, it remains to consider the case when $\mathbf u \approx \mathbf u^\prime$ is directly deducible from $\mathbf v_{0,6n} \approx \mathbf v_{2,6n}$, i.e., there exist some words $\mathbf a,\mathbf b \in \mathfrak X^\ast$ and substitution $\phi\colon \mathfrak X \to \mathfrak X^\ast$ such that $\{\mathbf u,\mathbf u^\prime \} = \{\mathbf a\phi(\mathbf v_{0,6n})\mathbf b,\mathbf a\phi(\mathbf v_{2,6n})\mathbf b\}$.

If $\phi(x)=1$, then $\phi(\mathbf v_{0,6n})=\phi(\mathbf v_{2,6n})$ and so $\mathbf u=\mathbf u^\prime$, whence $\mathbf u^\prime \in A$.
It remains to consider the case when $\phi(x)\ne1$.
Clearly, $t_i, t_i^\prime,t_i^{\prime\prime}\notin\con(\phi(x))$ for any $i=1,2,\dots,2n$ because $t_i, t_i^\prime,t_i^{\prime\prime}\in\simple(\mathbf u)$, while $x\in\mul(\mathbf v_{0,6n})=\mul(\mathbf v_{2,6n})$.
We note also that the letters of the form $z_i$, $z_i^\prime$ or $z_i^{\prime\prime}$ do not occur in $\phi(x)$ as well because the first and the second occurrences of these letters in $\mathbf u$ lie in different blocks, while the first and the second occurrences of $x$ in both $\mathbf v_{0,6n}$ and $\mathbf v_{2,6n}$ lie in the same block.
Therefore, $\con(\phi(x))=\{x\}$.
Since $\occ_x(\mathbf u)=\occ_x(\mathbf v_{0,6n})=\occ_x(\mathbf v_{2,6n})=2$, we have $\phi(x)=x$.
It follows that $x\notin \con(\phi(z_iz_i^\prime z_i^{\prime\prime}t_it_i^\prime t_i^{\prime\prime}))$ for any $i=1,2,\dots,2n$.
Clearly, $t_i, t_i^\prime,t_i^{\prime\prime}\notin\con(\phi(z_jz_j^\prime z_j^{\prime\prime}))$ for any $i,j=1,2,\dots,n$ because $t_i, t_i^\prime,t_i^{\prime\prime}\in\simple(\mathbf u)$, while $z_j,z_j^\prime, z_j^{\prime\prime}\in\mul(\mathbf v_{0,6n})=\mul(\mathbf v_{2,6n})$.
We note also that $\phi(z_i)$, $\phi(z_i^\prime)$ and $\phi(z_i^{\prime\prime})$ are either empty words or letters for any $i=1,2,\dots,n$ because any factor of $\mathbf u$ of length greater than~1 occurs only once in $\mathbf u$.
This implies that 
\begin{equation}
\label{phi(z_i) subseteq}
\{\phi(z_i),\phi(z_i^\prime),\phi(z_i^{\prime\prime})\mid 1\le i\le 2n\}\subseteq \{1\}\cup\{z_i,z_i^\prime,z_i^{\prime\prime}\mid 1\le i\le 2n\}.
\end{equation}
In view of the above, if $\mathbf u = \mathbf a\phi(\mathbf v_{2,6n})\mathbf b$, then $\mathbf u^\prime= \mathbf a\phi(\mathbf v_{0,6n})\mathbf b\in A$.
So, it remains to consider the case when $\mathbf u = \mathbf a\phi(\mathbf v_{0,6n})\mathbf b$ and $\mathbf u^\prime= \mathbf a\phi(\mathbf v_{2,6n})\mathbf b$.
Since $\mathbf u=\mathbf v_{\ell,2n}$ for some $\ell\in\{0,1,2\}$, we have $\phi(\mathbf q)=\mathbf q[\ell;6n-\ell]$ .

If $\ell=0$, then, since $\phi(z_{1\rho})$ and $\phi(z_{2\rho})$ are either empty words or letters by~\eqref{phi(z_i) subseteq}, we have $\mathbf u^\prime\in A$.

Suppose that $\ell=1$.
Then $\phi(z_{1\rho})= 1$ because $\phi(z_{1\rho})$ cannot coincide with $z_{2\rho}$.
If $\phi(z_{2\rho})=1$, then $\phi(\mathbf v_{0,6n})=\phi(\mathbf v_{2,6n})$ and so $\mathbf u=\mathbf u^\prime$, whence $\mathbf u^\prime \in A$.
If $\phi(z_{2\rho})\ne1$, then $\phi(z_{2\rho})=z_{2\rho}$ by~\eqref{phi(z_i) subseteq}.
It follows that $\mathbf u^\prime= \mathbf v_{2,6n}$ and so $\mathbf u^\prime\in A$.

Suppose now that $\ell=2$.
Then
\begin{equation}
\label{z_i,z_i',z_i''}
\{z_{i_1},z_{i_2}^\prime,z_{i_3}^{\prime\prime}\mid 1\le i_1,i_2,t_3\le n\}\setminus\{z_{1\rho}\}
\end{equation}
is a subset of 
$$
\{\phi(z_{i_1}),\phi(z_{i_2}^\prime),\phi(z_{i_3}^{\prime\prime})\mid 1\le i_1,i_2,i_3\le n\}
$$
by~\eqref{phi(z_i) subseteq} and, moreover, the sets~\eqref{z_i,z_i',z_i''} and $\{\phi(t_i),\phi(t_i^\prime),\phi(t_i^{\prime\prime})\mid 1\le i\le n\}$ are disjoint.
If $\phi(z_{1\rho})= 1$, then $\phi(z_{2\rho})= 1$ because $\phi(z_{2\rho})$ cannot coincide with $z_{1\rho}^\prime$.
In this case, $\phi(\mathbf v_{0,6n})=\phi(\mathbf v_{2,6n})$ and so $\mathbf u=\mathbf u^\prime$, whence $\mathbf u^\prime \in A$.
So, we may assume that $\phi(z_{1\rho})\ne1$.
Then $\phi(z_{1\rho})=z_{1\rho}^\prime$ by~\eqref{phi(z_i) subseteq}.
Two cases are possible:
\begin{itemize}
\item[\textup{(a)}] $1\rho < (2n-1)\rho$;
\item[\textup{(b)}] $(2n-1)\rho< 1\rho$.
\end{itemize}

Suppose that~(a) is true.
Then the set
$$
\{z_{i_1},z_{i_2}^\prime,z_{i_3}^{\prime\prime}\mid 1\le i_1\le n,\ 1\rho+1\le i_2,i_3 \le n\}\setminus\{z_{1\rho}\}
$$
must be a subset of 
$$
\{\phi(z_{i_1}),\phi(z_{i_2}^\prime),\phi(z_{i_3}^{\prime\prime})\mid 1\rho+1\le i_1,i_3\le n,\ (2n-1)\rho+1\le i_2\le n\}.
$$
But this is impossible because the first of these sets contains $(n-1)+2\cdot(n-1\rho)$ elements, the second one has at most $(n-(2n-1)\rho)+2\cdot(n-1\rho)$ elements and $1< (2n-1)\rho$ by~\eqref{1rho,(2n-1)rho,2rho,(2n)rho}.

Suppose now that~(b) is true.
Then the set
$$
\{z_{i_1},z_{i_2}^\prime,z_{i_3}^{\prime\prime}\mid 1\le i_1\le n,\ 1\le i_2,i_3 \le 1\rho\}\setminus\{z_{1\rho}\}
$$
must be a subset of 
$$
\{\phi(z_{i_1}),\phi(z_{i_2}^\prime),\phi(z_{i_3}^{\prime\prime})\mid 1\le i_1,i_3\le1\rho,\ 1\le i_2\le (2n-1)\rho\}.
$$
But this is also impossible because the first of these sets contains $(n-1)+2\cdot(1\rho)$ elements, the second one has at most $(2n-1)\rho+2\cdot(1\rho)$ elements and $(2n-1)\rho<1\rho<n$ by~\eqref{1rho,(2n-1)rho,2rho,(2n)rho}.

We see that $\mathbf u^\prime \in A$ in either case.
This means that $A$ forms a $\FIC(\mathbf X)$-class.
By similar arguments we can show that $\{\mathbf v_{0,6n},\mathbf v_{0,6n-1},\mathbf v_{0,6n-2}\}$ is a $\FIC(\mathbf Y)$-class.
It follows that the word $\mathbf v_{0,6n}$ is an isoterm for $\mathbf X\vee \mathbf Y$.
Put $\mathbf Z = \mathbf M(\mathbf v_{0,6n})=\mathbf M(\mathbf a_{3n,3n}[\tau])$.
According to Lemma~\ref{L: M(W) in V}, $\mathbf Z\subseteq \mathbf X\vee \mathbf Y$ and, therefore, $(\mathbf X\vee \mathbf Y)\wedge \mathbf Z=\mathbf Z$.
It is routine to check that $\mathbf Z$ satisfies $\mathbf a_{k,k}[\pi]\approx \mathbf a_{k,k}^\prime[\pi]$ for any $k=1,2,\dots,3n-1$ and any $\pi\in S_{k,k}$.
Then the identity $\mathbf v_{0,6n}\approx \mathbf v_{0,0}$ is satisfied by $\mathbf X\wedge \mathbf Z$ because $\mathbf v_{0,6n}\approx \mathbf v_{2,6n}$ 
holds in $\mathbf X$, $\mathbf v_{2,6n}\approx \mathbf v_{2,2}$ follows from some identity of the form $\mathbf a_{3n-1,3n-1}[\pi]\approx \mathbf a_{3n-1,3n-1}^\prime[\pi]$ and $\mathbf v_{2,2} \approx \mathbf v_{0,0}$ is a consequence of $x^2y\approx yx^2$.
By a similar argument we can show that the identity $\mathbf v_{0,6n}\approx \mathbf v_{0,0}$ is satisfied by $\mathbf Y\wedge \mathbf Z$ as well.
Therefore, the variety $(\mathbf X\wedge \mathbf Z)\vee(\mathbf Y\wedge \mathbf Z)$ satisfies $\mathbf v_{0,6n}\approx \mathbf v_{0,0}$.
We see that $\mathbf v_{0,6n}$ is not an isoterm for this variety, whence
$$
(\mathbf X\wedge \mathbf Z)\vee(\mathbf Y\wedge \mathbf Z)\subset \mathbf Z= (\mathbf X\vee \mathbf Y)\wedge \mathbf Z.
$$
Thus, we have proved that the lattice $\mathfrak L(\mathbf M(\mathbf a_{n,m}[\rho]))$ is not distributive.
\end{proof}

\subsection{Varieties of the form $\mathbf M(\mathbf c_{n,m,k}[\rho])$}
\label{Subsec: M(c_{n,m,k}[rho])}

\begin{proposition}
\label{P: L(M(c_{n,m,n+m+1}[rho])) is not modular}
The lattice $\mathfrak L(\mathbf M(\mathbf c_{n,m,n+m+1}[\rho]))$ is not modular and so not distributive for any $n,m\in \mathbb N_0$ and $\rho\in S_{n+m,n+m+1}$.
\end{proposition}

\begin{proof}
For any $p\in\mathbb N$, $\theta\in S_p$ and $q,r\in\{1,2,\dots, p+1\}$, we define the permutation $\theta_{q,r}\in S_{p+1}$ as follows: $q\theta_{q,r} = r$, for any $i=1,2,\dots,q-1$,
$$ 
i\theta_{q,r} = 
\begin{cases} 
i\theta & \text{if }i\theta< r, \\ 
i\theta+1 & \text{if }r\le i\theta, 
\end{cases} 
$$ 
and, for any $i=q+1,q+2,\dots,p+1$,
$$ 
i\theta_{q,r} = 
\begin{cases} 
(i-1)\theta & \text{if }(i-1)\theta< r, \\ 
(i-1)\theta+1 & \text{if }r\le (i-1)\theta. 
\end{cases} 
$$ 

Put $k=n+m+1$ and $\pi=(\rho_{1,k})_{2k+1,2k+1}$.
Let $\mathbf c_{n,m,k+2}[\pi]\approx \mathbf c$ be an identity that holds in $\mathbf M(\mathbf c_{n,m,k+2}[\rho])$.
Clearly, the word $xyx$ is an isoterm for $\mathbf M(\mathbf c_{n,m,k}[\rho])$.
Then Lemma~\ref{L: xt_1x...t_kx is isoterm} implies that 
$$
\mathbf c=\biggl(\prod_{i=1}^n z_it_i\biggr)\mathbf c^\prime t\biggl(\prod_{i=n+1}^{n+m} z_it_i\biggr)\mathbf c^{\prime\prime}\biggl(\prod_{i=n+m+1}^{2k+1} t_iz_i\biggr),
$$ 
where $\mathbf c^\prime$ and $\mathbf c^{\prime\prime}$ are linear words with 
$$
\con(\mathbf c^\prime)=\{x,y\}\  \text{ and } \ \con(\mathbf c^{\prime\prime})=\{x,y,z_1,z_2,\dots,z_{2k+1}\}. 
$$
Let $1\le p\le 2k$.
Evidently, $xyzxty$ is an isoterm for $\mathbf M(\mathbf c_{n,m,k}[\rho])$.
Then if $n+m<p\pi,(p+1)\pi$, then $({_{1\mathbf c^{\prime\prime}}}z_{p\pi})<({_{1\mathbf c^{\prime\prime}}}z_{(p+1)\pi})$.
Further, since $k>0$, one can show that a non-trivial identity of the form $xzxyty\approx \mathbf v$ implies some non-trivial identity of the form $\mathbf c_{n,m,k}[\rho] \approx \mathbf w$, whence $xzxyty$ is an isoterm for $\mathbf M(\mathbf c_{n,m,k}[\rho])$.
Hence if $p\pi\le n+m<(p+1)\pi$ or $(p+1)\pi\le n+m<p\pi$, then $({_{1\mathbf c^{\prime\prime}}}z_{p\pi})<({_{1\mathbf c^{\prime\prime}}}z_{(p+1)\pi})$.
The case when $p\pi,(p+1)\pi\le n+m$ is impossible by definition of $\pi$ and the fact that $\rho\in S_{n+m,n+m+1}$.
We see that $({_{1\mathbf c^{\prime\prime}}}z_{p\pi})<({_{1\mathbf c^{\prime\prime}}}z_{(p+1)\pi})$ in either case.
By a similar argument we can show that $({_{1\mathbf c^{\prime\prime}}}x)<({_{1\mathbf c^{\prime\prime}}}z_{1\pi})$ and $({_{1\mathbf c^{\prime\prime}}}z_{(2k+1)\pi})<({_{1\mathbf c^{\prime\prime}}}y)$.
Hence $\mathbf c^{\prime\prime}=xz_{1\pi}z_{2\pi}\cdots z_{(2k+1)\pi}y$.

Further, by definition, $(\mathbf c_{n,m,k+2}[\pi])_X$ coincides (up to renaming of letters) with $\mathbf c_{n,m,k}[\rho]$ for $X=\{z_{1\pi},z_{(2k+1)\pi},t_{1\pi},t_{(2k+1)\pi}\}$.
Then Lemma~\ref{L: M(W) in V} implies that $\mathbf c_X=(\mathbf c_{n,m,k+2}[\pi])_X$.
In particular, $\mathbf c^\prime=xy$.
Thus, $\mathbf c=\mathbf c_{n,m,k+2}[\pi]$. 
We see that $\mathbf c_{n,m,k+2}[\pi]$ is an isoterm for $\mathbf M(\mathbf c_{n,m,k}[\rho])$.
By similar arguments we can show that if $\tau=(\rho_{1,2k})_{2k+1,k}$, then $\mathbf c_{n,m,k+2}[\tau]$ is an isoterm for $\mathbf M(\mathbf c_{n,m,k}[\rho])$.
Then $M(\mathbf c_{n,m,k+2}[\pi])$ and $M(\mathbf c_{n,m,k+2}[\tau])$ lie in $\mathbf M(\mathbf c_{n,m,k}[\rho])$.
It is routine to check that the monoids $M(\mathbf c_{n,m,k+2}[\pi])$ and $M(\mathbf c_{n,m,k+2}[\tau])$ satisfy the identities $\mathbf c_{n,m,k+2}[\tau] \approx \mathbf c_{n,m,k+2}^\prime[\tau]$ and $\mathbf c_{n,m,k+2}[\pi] \approx \mathbf c_{n,m,k+2}^\prime[\pi]$, respectively.
So, we have proved that the variety $\mathbf M(\mathbf c_{n,m,k}[\rho])$ contains two incomparable subvarieties $\mathbf M(\mathbf c_{n,m,k+2}[\pi])$ and $\mathbf M(\mathbf c_{n,m,k+2}[\tau])$.
In view of this fact and Lemma~5.1 in~\cite{Jackson-Sapir-00}, it suffices to show that the lattice $\mathfrak L(\mathbf M(\mathbf c_{n,m,k+2}[\pi],\mathbf c_{n,m,k+2}[\tau]))$ is not modular.

For any $\xi,\eta \in S_2$, we define the word:
$$
\mathbf v_{\xi,\eta}= \mathbf p\,a_1b_1\, x_{1\xi}x_{2\xi}\, y_{1\eta}y_{2\eta}\, b_2a_2\,\mathbf q\mathbf r\mathbf s,
$$
where
$$
\begin{aligned}
&\mathbf p = \biggl(\prod_{i=1}^n z_it_i\biggr)\biggl(\prod_{i=1}^n z_i^\prime t_i^\prime\biggr)\biggl(\prod_{i=1}^n z_i^{\prime\prime} t_i^{\prime\prime}\biggr),\\
&\mathbf q = t\biggl(\prod_{i=n+1}^{k-1} z_it_i\biggr)\biggl(\prod_{i=n+1}^{k-1} z_i^\prime t_i^\prime\biggr)\biggl(\prod_{i=n+1}^{k-1} z_i^{\prime\prime} t_i^{\prime\prime}\biggr),\\
&\mathbf r =x_1z_{1\pi}^\prime a_1\biggl(\prod_{i=2}^{2k} z_{i\pi}^\prime\biggr) b_2z_{(2k+1)\pi}^\prime x_2 \biggl(\prod_{i=1}^{2k+1}z_{i\tau}\biggr) y_1z_{1\pi}^{\prime\prime} b_1\biggl(\prod_{i=2}^{2k} z_{i\pi}^{\prime\prime}\biggr)a_2z_{(2k+1)\pi}^{\prime\prime} y_2,\\
&\mathbf s = \biggl(\prod_{i=k}^{2k+1} t_iz_i\biggr)\biggl(\prod_{i=k}^{2k+1} t_i^\prime z_i^\prime\biggr)\biggl(\prod_{i=k}^{2k+1} t_i^{\prime\prime} z_i^{\prime\prime}\biggr).
\end{aligned}
$$

Let $\varepsilon$ denote the trivial permutation from $S_2$.
We need the following two auxiliary facts.

\begin{lemma}
\label{L: FIC(M(c_{n,m,k+2}[tau]))-class}
The set $\{\mathbf v_{\xi,\eta}\mid \xi,\eta\in S_2\}$ forms a $\FIC(\mathbf M(\mathbf c_{n,m,k+2}[\tau]))$-class.\qed
\end{lemma}

\begin{proof}[Proof of Lemma~\ref{L: FIC(M(c_{n,m,k+2}[tau]))-class}]
It is routine to check that $M(\mathbf c_{n,m,k+2}[\tau])$ satisfies $\mathbf v_{\xi_1,\eta_1}\approx \mathbf v_{\xi_2,\eta_2}$ for any $\xi_1,\eta_1,\xi_2,\eta_2\in S_2$.
Let now $\mathbf v_{\varepsilon,\varepsilon}\approx \mathbf v$ be an identity of $M(\mathbf c_{n,m,k+2}[\tau])$.
Clearly, $xyx$ is an isoterm for $M(\mathbf c_{n,m,k+2}[\tau])$. 
Then Lemma~\ref{L: xt_1x...t_kx is isoterm} implies that $\mathbf v=\mathbf p\mathbf v^\prime \mathbf q\mathbf r^\prime$, where $\mathbf v^\prime$ and $\mathbf r^\prime$ are linear words with 
$$
\con(\mathbf v^\prime)=\{a_1,a_2,b_1,b_2,x_1,x_2,y_1,y_2\}\  \text{ and } \ \con(\mathbf r)=\con(\mathbf r^\prime). 
$$
Further, arguments similar to ones from the second paragraph of the proof of Proposition~\ref{P: L(M(c_{n,m,n+m+1}[rho])) is not modular} imply that all the letters occur in $\mathbf r^\prime$ in the same order as in $\mathbf r$ and so $\mathbf r^\prime=\mathbf r$.
Since $(\mathbf v_{\varepsilon,\varepsilon})_X$ coincides (up to renaming of letters) with $\mathbf c_{n,m,k+2}[\tau]$ for 
$$
X=\{a_1,b_1,t\}\cup\{z_i,t_i\mid 1\le i\le 2k+1\},
$$ 
Lemma~\ref{L: M(W) in V} implies that $(_{1\mathbf v}a_1)<(_{1\mathbf v}b_1)$.
By a similar argument we can show that 
$$
\begin{aligned}
&(_{1\mathbf v}b_1)<(_{1\mathbf v}x_1),\ (_{1\mathbf v}b_1)<(_{1\mathbf v}x_2),\ (_{1\mathbf v}x_1)<(_{1\mathbf v}y_1),\\
&(_{1\mathbf v}x_1)<(_{1\mathbf v}y_2),\ (_{1\mathbf v}x_2)<(_{1\mathbf v}y_1),\ (_{1\mathbf v}x_2)<(_{1\mathbf v}y_2),\\
&(_{1\mathbf v}y_1)<(_{1\mathbf v}b_2),\ (_{1\mathbf v}y_2)<(_{1\mathbf v}b_2),\
(_{1\mathbf v}b_2)<(_{1\mathbf v}a_2).
\end{aligned}
$$
It follows that $\mathbf v =\mathbf v_{\xi,\eta}$ for some $\xi,\eta\in S_2$.
Therefore, the set $\{\mathbf v_{\xi,\eta}\mid \xi,\eta\in S_2\}$ forms a $\FIC(\mathbf M(\mathbf c_{n,m,k+2}[\tau]))$-class.
\end{proof}

\begin{lemma}
\label{L: v_{xi,eta}}
Let $\xi_1,\eta_1,\xi_2,\eta_2\in S_2$. 
If a non-trivial identity $\mathbf v_{\varepsilon,\varepsilon} \approx \mathbf v$ is directly deducible from the identity $\mathbf v_{\xi_1,\eta_1} \approx \mathbf v_{\xi_2,\eta_2}$, then $\{\mathbf v_{\varepsilon,\varepsilon},\mathbf v\}=\{\mathbf v_{\xi_1,\eta_1},\mathbf v_{\xi_2,\eta_2}\}$.
\end{lemma}

\begin{proof}[Proof of Lemma~\ref{L: v_{xi,eta}}]
Since $\mathbf v_{\varepsilon,\varepsilon} \approx \mathbf v$ is directly deducible from $\mathbf v_{\xi_1,\eta_1} \approx \mathbf v_{\xi_2,\eta_2}$, there are words $\mathbf a,\mathbf b\in \mathfrak X^\ast$ and an endomorphism $\phi$ of $\mathfrak X^\ast$ such that $\mathbf v_{\varepsilon,\varepsilon}=\mathbf a\phi(\mathbf v_{\xi_1,\eta_1})\mathbf b$ and $\mathbf v=\mathbf a\phi(\mathbf v_{\xi_2,\eta_2})\mathbf b$.
Further, Lemma~\ref{L: FIC(M(c_{n,m,k+2}[tau]))-class} implies that $\mathbf v=\mathbf v_{\xi,\eta}$ for some $\xi,\eta\in S_2$.
Since the identity $\mathbf v_{\varepsilon,\varepsilon} \approx \mathbf v$ is non-trivial, $(\xi,\eta)\ne(\varepsilon,\varepsilon)$.
By symmetry, we may assume that $\xi\ne\varepsilon$.
Then $(_{1\mathbf v}x_2)<(_{1\mathbf v}x_1)$.
This is only possible when one of the following holds:
\begin{itemize}
\item $\xi_1\ne\xi_2$, $x_1\in\con(\phi(x_{1\xi_1}))$ and $x_2\in\con(\phi(x_{2\xi_1}))$;
\item $\eta_1\ne\eta_2$, $x_1\in\con(\phi(y_{1\eta_1}))$ and $x_2\in\con(\phi(y_{2\eta_1}))$.
\end{itemize}
We note that every factor of length $>1$ of the word $\mathbf v_{\varepsilon,\varepsilon}$ has exactly one occurrence in this word.
It follows that
\begin{itemize}
\item[\textup{($\ast$)}] $\phi(c)$ is either the empty word or a letter for any $c\in\mul(\mathbf v_{\xi_1,\eta_1})$.
\end{itemize}
In view of this fact, one of the following holds:
\begin{itemize}
\item[\textup{(a)}] $\xi_1\ne\xi_2$, $\phi(x_{1\xi_1})=\phi(x_{2\xi_2})=x_1$ and $\phi(x_{2\xi_1})=\phi(x_{1\xi_2})=x_2$;
\item[\textup{(b)}] $\eta_1\ne\eta_2$, $\phi(y_{1\eta_1})=\phi(y_{2\eta_2})=x_1$ and $\phi(y_{2\eta_1})=\phi(y_{1\eta_2})=x_2$.
\end{itemize}

Suppose that~(a) holds.
Since $(_{2\mathbf v_{\varepsilon,\varepsilon}}x_1)<(_{2\mathbf v_{\varepsilon,\varepsilon}}x_2)$, we have $(_{2\mathbf v_{\xi_1,\eta_1}}x_{1\xi_1})<(_{2\mathbf v_{\xi_1,\eta_1}}x_{2\xi_1})$.
This implies that $1\xi_1=1$ and $2\xi_1=2$, whence $\xi_1=\varepsilon$.
Then
$$
\phi\biggl(z_{1\pi}^\prime a_1\biggl(\prod_{i=2}^{2k} z_{i\pi}^\prime\biggr) b_2z_{(2k+1)\pi}^\prime\biggr)=z_{1\pi}^\prime a_1\biggl(\prod_{i=2}^{2k} z_{i\pi}^\prime\biggr) b_2z_{(2k+1)\pi}^\prime.
$$
It follows from~($\ast$) that $\phi(a_1)=a_1$, $\phi(b_2)=b_2$ and $\phi(z_{i\pi}^\prime)=z_{i\pi}^\prime$ for any $i=1,2,\dots,2k+1$.
Then 
$$
\phi(b_1x_{1\xi_1}x_{2\xi_1}y_{1\eta_1}y_{2\eta_1})=b_1x_1x_2y_1y_2.
$$
Now we apply~($\ast$) again and obtain that $\phi(b_1)=b_1$ and $\phi(y_i)=y_{i\eta_1}$ for any $i=1,2$.
Since $(_{2\mathbf v_{\varepsilon,\varepsilon}}y_1)<(_{2\mathbf v_{\varepsilon,\varepsilon}}y_2)$, we have $(_{2\mathbf v_{\xi_1,\eta_1}}y_{1\eta_1})<(_{2\mathbf v_{\xi_1,\eta_1}}y_{2\eta_1})$.
This implies that $1\eta_1=1$ and $2\eta_1=2$, whence $\eta_1=\varepsilon$.
Then $\mathbf v= \mathbf v_{\xi_2,\eta_2}$ and so $\{\mathbf v_{\varepsilon,\varepsilon},\mathbf v\}=\{\mathbf v_{\xi_1,\eta_1},\mathbf v_{\xi_2,\eta_2}\}$.

Suppose that~(b) holds.
Since $(_{2\mathbf v_{\varepsilon,\varepsilon}}x_1)<(_{2\mathbf v_{\varepsilon,\varepsilon}}x_2)$, we have $(_{2\mathbf v_{\xi_1,\eta_1}}y_{1\eta_1})<(_{2\mathbf v_{\xi_1,\eta_1}}y_{2\eta_1})$.
This implies that $1\eta_1=1$ and $2\eta_1=2$, whence $\eta_1=\varepsilon$.
Then
$$
\phi\biggl(z_{1\pi}^{\prime\prime} b_1\biggl(\prod_{i=2}^{2k} z_{i\pi}^{\prime\prime}\biggr) a_2z_{(2k+1)\pi}^{\prime\prime}\biggr)=z_{1\pi}^\prime a_1\biggl(\prod_{i=2}^{2k} z_{i\pi}^\prime\biggr) b_2z_{(2k+1)\pi}^\prime.
$$
It follows from~($\ast$) that $\phi(b_1)=a_1$ and $\phi(a_2)=b_2$.
Then 
$$
\phi(x_{1\xi_1}x_{2\xi_1}y_{1\eta_1}y_{2\eta_1}b_2)=b_1x_1x_2y_1y_2.
$$
Now we apply~($\ast$) again and obtain that $\phi(y_{2\eta_1})=y_1$, which contradicts~(b).
Therefore,~(b) is impossible.

Lemma~\ref{L: v_{xi,eta}} is proved.
\end{proof}

One can return to the proof of Proposition~\ref{P: L(M(c_{n,m,n+m+1}[rho])) is not modular}.
Let 
$$
\begin{aligned}
&\mathbf X = \mathbf M(\mathbf c_{n,m,k+2}[\pi],\mathbf c_{n,m,k+2}[\tau])\wedge\var\{\mathbf v_1 \approx \mathbf v_2,\,\mathbf v_3 \approx \mathbf v_4\},\\
&\mathbf Y = \mathbf M(\mathbf c_{n,m,k+2}[\pi],\mathbf c_{n,m,k+2}[\tau])\wedge\var\{\mathbf v_2 \approx \mathbf v_4\},\\
&\mathbf Z = \mathbf M(\mathbf c_{n,m,k+2}[\pi],\mathbf c_{n,m,k+2}[\tau])\wedge\var\{\mathbf v_1 \approx \mathbf v_3,\,\mathbf v_2 \approx \mathbf v_4\}.
\end{aligned}
$$ 
Consider an identity $\mathbf u \approx \mathbf u^\prime$ of $\mathbf X$ with $\mathbf u\in \{\mathbf v_1, \mathbf v_2\}$.
We are going to show that $\mathbf u^\prime \in \{\mathbf v_1, \mathbf v_2\}$.
In view of Proposition~\ref{P: deduction}, we may assume without loss of generality that  either $\mathbf u \approx \mathbf u^\prime$ holds in $\mathbf M(\mathbf c_{n,m,k+2}[\pi],\mathbf c_{n,m,k+2}[\tau])$ or $\mathbf u \approx \mathbf u^\prime$ is directly deducible from $\mathbf v_1 \approx \mathbf v_2$ or $\mathbf v_3 \approx \mathbf v_4$. 
According to Lemma~\ref{L: v_{xi,eta}}, $\mathbf u \approx \mathbf u^\prime$ cannot be directly deducible from $\mathbf v_3 \approx \mathbf v_4$ and if $\mathbf u \approx \mathbf u^\prime$ is directly deducible from $\mathbf v_1 \approx \mathbf v_2$, then $\{\mathbf u, \mathbf u^\prime\}\subseteq \{\mathbf v_1, \mathbf v_2\}$.
Therefore, it remains to consider the case when $\mathbf u \approx \mathbf u^\prime$ is satisfied by $\mathbf M(\mathbf c_{n,m,k+2}[\pi],\mathbf c_{n,m,k+2}[\tau])$.
It follows from Lemma~\ref{L: FIC(M(c_{n,m,k+2}[tau]))-class} that $\mathbf u^\prime\in\{\mathbf v_1,\mathbf v_2,\mathbf v_3,\mathbf v_4\}$.
If $\mathbf u^\prime\in\{\mathbf v_3,\mathbf v_4\}$, then the identity $\mathbf u(X)\approx \mathbf u^\prime(X)$ coincides (up to renaming of letters) with the identity $\mathbf c_{n,m,k+2}[\pi]\approx \mathbf c_{n,m,k+2}^\prime[\pi]$, where
$$
X=\{y_1,y_2,t\}\cup\{z_i^{\prime\prime},t_i^{\prime\prime}\mid1\le i\le 2k+1\}.
$$
But this is impossible because the variety $\mathbf M(\mathbf c_{n,m,k+2}[\pi],\mathbf c_{n,m,k+2}[\tau])$ violates this identity by Lemma~\ref{L: M(W) in V}.
We see that $\mathbf u^\prime\in \{\mathbf v_1, \mathbf v_2\}$ in either case and so the set $\{\mathbf v_1,\mathbf v_2\}$ forms a $\FIC(\mathbf X)$-class.
By similar arguments we can show that the set $\{\mathbf v_1,\mathbf v_3\}$ forms a $\FIC(\mathbf Z)$-class and the word $\mathbf v_1$ is an isoterm for $\mathbf Y$.
This implies that $\mathbf v_1$ is an isoterm for $(\mathbf X\vee \mathbf Z)\wedge \mathbf Y$.
Clearly, both $\mathbf X\wedge \mathbf Y$ and $\mathbf Z$ satisfy the identity $\mathbf v_1\approx \mathbf v_3$.
Therefore, $\mathbf v_1$ is not an isoterm for $(\mathbf X\wedge \mathbf Y)\vee\mathbf Z$.
Since $\mathbf Z\subseteq\mathbf Y$, we have
$$
(\mathbf X\wedge \mathbf Y)\vee\mathbf Z\subset (\mathbf X\vee \mathbf Z)\wedge \mathbf Y.
$$
It follows that the lattice $\mathfrak L(\mathbf M(\mathbf c_{n,m,n+m+1}[\rho]))$ is not modular.
\end{proof}

The proof of the following statement is very similar to the proof of Proposition~\ref{P: L(M(c_{n,m,n+m+1}[rho])) is not modular} and so we omit it.

\begin{proposition}
\label{P: L(M(c_{n,m,0}[rho])) is not modular}
The lattice $\mathfrak L(\mathbf M(\mathbf c_{n,m,0}[\rho]))$ is not modular and so not distributive for any $m,m\in \mathbb N_0$ and $\rho\in S_{n+m}$.\qed
\end{proposition}

\section{Auxiliary results}
\label{Sec: auxiliary results}

\subsection{Identities formed by words with one multiple letter}

Let 
$$
\mathbf A=\var\{x^2y\approx yx^2\}.
$$

\begin{lemma}
\label{L: one letter in block reduction}
Let $r,e_0,f_0,e_1,f_1,\dots,e_r,f_r\in\mathbb N_0$.
Then the identity
\begin{equation}
\label{one letter in a block}
x^{e_0}\biggl(\prod_{i=1}^r t_ix^{e_i}\biggr) \approx x^{f_0}\biggl(\prod_{i=1}^r t_ix^{f_i}\biggr),
\end{equation}
is equivalent within $\mathbf A$ to some identity
$$
x^{e_0^\prime}\biggl(\prod_{i=1}^r t_ix^{e_i^\prime}\biggr) \approx x^{f_0^\prime}\biggl(\prod_{i=1}^r t_ix^{f_i^\prime}\biggr)
$$
with $e_1^\prime,f_1^\prime,e_2^\prime,f_2^\prime,\dots,e_r^\prime,f_r^\prime\le1$.
\end{lemma}

\begin{proof}
Let $p_i$ and $q_i$ be the greatest even numbers such that $p_i\le e_i$ and $q_i\le f_i$ for any $i=1,2,\dots,r$.
Then the required conclusion follows from the fact that the identities
$$
x^{e_0}\biggl(\prod_{i=1}^r t_ix^{e_i}\biggr) \approx x^{e_0+\sum_{i=1}^rp_i}\biggl(\prod_{i=1}^r t_ix^{e_i-p_i}\biggr)\ \text{ and }\ 
x^{f_0}\biggl(\prod_{i=1}^r t_ix^{f_i}\biggr) \approx x^{f_0+\sum_{i=1}^rq_i}\biggl(\prod_{i=1}^r t_ix^{f_i-q_i}\biggr)
$$
are consequences of $x^2y\approx yx^2$.
\end{proof}

For any $n,m\in\mathbb N_0$, we fix notation for the following identity:
$$
\begin{aligned}
\delta_{n,m}&:\enskip x^mt_1xt_2x\cdots t_nx\approx x^{m+n}t_1t_2\cdots t_n.
\end{aligned}
$$

\begin{lemma}
\label{L: reduction to delta_{r,m}}
Let $\mathbf V$ be an aperiodic monoid subvariety of $\mathbf A$, $0\le e_1,f_1,\dots,e_n,f_r\le 1$ and $e_0,f_0\in\mathbb N_0$. 
If the identity~\eqref{one letter in a block} is non-trivial, then
$$
\mathbf V\{\eqref{one letter in a block}\}=\mathbf V\{x^e\approx x^f,\,\delta_{e-e_0,e_0},\,\delta_{f-f_0,f_0}\},
$$
where $e=\sum_{i=0}^re_i$ and $f=\sum_{i=0}^rf_i$.
\end{lemma}

\begin{proof}
The identity~\eqref{one letter in a block} follows from $\{x^e\approx x^f,\,\delta_{e-e_0,e_0},\,\delta_{f-f_0,f_0}\}$ because
$$
x^{e_0}\biggl(\prod_{i=1}^r t_ix^{e_i}\biggr) \stackrel{\delta_{e-e_0,e_0}}\approx x^e\biggl(\prod_{i=1}^r t_i\biggr)\stackrel{x^e\approx x^f}\approx x^f\biggl(\prod_{i=1}^r t_i\biggr)\stackrel{\delta_{f-f_0,f_0}}\approx x^{f_0}\biggl(\prod_{i=1}^r t_ix^{f_i}\biggr).
$$
Evidently, $x^e\approx x^f$ is a consequence of~\eqref{one letter in a block}.
So, it remains to show that $\mathbf V\{\eqref{one letter in a block}\}$ satisfies $\delta_{e-e_0,e_0}$ and $\delta_{f-f_0,f_0}$.
Two cases are possible.

\medskip

\textit{Case }1: $e_i=f_i=1$ for any $i=1,2,\dots,r$.
Then $r=e-e_0=f-f_0$.
Since the identity~\eqref{one letter in a block} is non-trivial, by symmetry, we may assume that $e_0=f_0+g$ for some $g\in\mathbb N$.
In view of Lemma~\ref{L: subvariety of A_cen}, $\mathbf V$ satisfies the identity $x^s \approx x^{s+1}$ for some $s\in\mathbb N$.
Then the identities
$$
x^{e_0}t_1xt_2x\cdots t_rx \stackrel{\eqref{one letter in a block}}\approx x^{e_0+sg}t_1xt_2x\cdots t_rx \stackrel{\{x^s \approx x^{s+1},\,x^2y\approx yx^2\}}\approx x^{e_0+r+sg} t_1t_2\cdots t_r \stackrel{\eqref{one letter in a block}}\approx x^{e_0+r} t_1t_2\cdots t_r
$$
and so $\delta_{r,e_0}$ are satisfied by $\mathbf V\{\eqref{one letter in a block}\}$.
It remains to notice that $\delta_{r,f_0}$ is consequence of $\delta_{r,e_0}$.

\medskip

\textit{Case }2: $(e_j,f_j)\ne (1,1)$ for some $j\in\{1,2,\dots,r\}$.
We use induction on $r$.

\smallskip

\textit{Induction base}: $r=1$. 
We may assume without any loss that~\eqref{one letter in a block} coincides with either $x^{e_0}t_1\approx x^{f_0}t_1$ or $x^{e_0}t_1\approx x^{f_0}t_1x$.
If~\eqref{one letter in a block} coincides with $x^{e_0}t_1\approx x^{f_0}t_1$, then $e=e_0$ and $f=f_0$.
In this case, the identities $\delta_{e-e_0,e_0}$ and $\delta_{f-f_0,f_0}$ are trivial and so follow from~\eqref{one letter in a block}.
If~\eqref{one letter in a block} is equal to $x^{e_0}t_1\approx x^{f_0}t_1x$, then $\delta_{e-e_0,e_0}$ is trivial, while $\delta_{f-f_0,f_0}$ follows from~\eqref{one letter in a block} because 
$$
x^{f_0}t_1x \stackrel{\eqref{one letter in a block}}\approx x^{e_0}t_1 = x^et_1 \stackrel{x^e\approx x^f}\approx x^ft_1.
$$
We see that~\eqref{one letter in a block} implies $\delta_{e-e_0,e_0}$ and $\delta_{f-f_0,f_0}$ in either case.

\smallskip

\textit{Induction step}: $r>1$.
For any $\ell=1,2,\dots,r$, we put
$$
\mathbf p_\ell =\biggl(\prod_{i=1}^r t_ix^{e_i}\biggr)_{t_\ell}\ \text { and }\  \mathbf q_\ell =\biggl(\prod_{i=1}^r t_ix^{e_i}\biggr)_{t_\ell}.
$$

Suppose that $e_k+e_{k+1},f_k+f_{k+1}\le 1$ for some $k\in\{1,2,\dots,r-1\}$.
Consider the identity $x^{e_0}\mathbf p_{k+1} \approx x^{f_0}\mathbf q_{k+1}$.
If $e_i=f_i=e_k+e_{k+1}=f_k+f_{k+1}=1$ for any $i=1,2,\dots, k-1,k+2,\dots,r$, then $\mathbf V\{x^{e_0}\mathbf p_{k+1} \approx x^{f_0}\mathbf q_{k+1}\}$ satisfies the identities $\delta_{e-e_0-1,e_0}$ and $\delta_{f-f_0-1,f_0}$ by Case~1 and so the identities $\delta_{e-e_0,e_0}$ and $\delta_{f-f_0,f_0}$.
If either $(e_k+e_{k+1},f_k+f_{k+1})\ne(1,1)$ or $(e_q,f_q)\ne(1,1)$ for some $q\in\{1,2,\dots, k-1,k+2,\dots,r\}$, then $\delta_{e-e_0,e_0}$ and $\delta_{f-f_0,f_0}$ hold in $\mathbf V\{x^{e_0}\mathbf p_{k+1} \approx x^{f_0}\mathbf q_{k+1}\}$ by the induction assumption.
Since $x^{e_0}\mathbf p_{k+1} \approx x^{f_0}\mathbf q_{k+1}$ is consequence of~\eqref{one letter in a block}, the identities $\delta_{e-e_0,e_0}$ and $\delta_{f-f_0,f_0}$ are satisfied by $\mathbf V\{\eqref{one letter in a block}\}$.
So, we may further assume that $e_i+e_{i+1}>1$ or $f_i+f_{i+1}>1$ for any $i=1,2,\dots,r-1$.
In particular, $(e_i,f_i)\ne(0,0)$ for any $i=1,2,\dots,r$.

Then $(e_j,f_j)\in\{(0,1),(1,0)\}$. 
We may assume without any loss that $(e_j,f_j)=(0,1)$ and $j\ne r$. 
Then $f_{j+1}=1$.
Consider the identity $x^{e_0}\mathbf p_{j+1} \approx x^{f_0}\mathbf q_{j+1}$.
Clearly, $x^2y\approx yx^2$ implies $x^{f_0}\mathbf q_{j+1} \approx x^{f_0+2}\mathbf q_{j+1}^\prime$, where
$$
\mathbf q_{j+1}^\prime=\biggl(\prod_{i=1}^{j-1} t_ix^{f_i}\biggr)\cdot t_j\cdot\biggl(\prod_{i=j+2}^r t_ix^{f_i}\biggr).
$$
By the induction assumption, $\mathbf V\{x^{e_0}\mathbf p_{j+1} \approx x^{f_0+2}\mathbf q_{j+1}^\prime\}$ and so $\mathbf V\{\eqref{one letter in a block}\}$ satisfy $\delta_{e-e_0,e_0}$ and $\delta_{f-(f_0+2),f_0+2}$.
Then $\delta_{f-f_0,f_0}$ holds in $\mathbf V\{\eqref{one letter in a block}\}$ because 
$$
x^{f_0}\biggl(\prod_{i=1}^r t_ix^{f_i}\biggr) \stackrel{\eqref{one letter in a block}}\approx x^{e_0}\biggl(\prod_{i=1}^r t_ix^{e_i}\biggr) \stackrel{\delta_{e-e_0,e_0}}\approx x^e\biggl(\prod_{i=1}^r t_i\biggr) \stackrel{x^e\approx x^f}\approx x^f\biggl(\prod_{i=1}^r t_i\biggr).
$$

Lemma~\ref{L: reduction to delta_{r,m}} is proved.
\end{proof}

\begin{corollary}
\label{C: delta_{n,m} in X wedge Y}
Let $\mathbf X$ and $\mathbf Y$ be aperiodic monoid subvarieties of $\mathbf A$. 
If the variety $\mathbf X\wedge\mathbf Y$ satisfies the identity $\delta_{n,m}$ for some $n,m\in\mathbb N_0$, then this identity holds in either $\mathbf X$ or $\mathbf Y$.
\end{corollary}

\begin{proof}
If $M(xy)\notin\mathbf X$, then $\mathbf X$ is commutative by Lemma~\ref{L: does not contain M(xy)}. 
Then $\mathbf X$ satisfies the identity $\delta_{n,m}$ because it is a consequence of the commutative law.
By a similar argument we can show that if $M(xy)\notin\mathbf Y$, then $\delta_{n,m}$ holds in $\mathbf Y$.
So, we may further assume that $M(xy)\in\mathbf X\wedge\mathbf Y$.

It follows from Proposition~\ref{P: deduction} that there is a sequence of pairwise distinct words $\mathbf w_1,\mathbf w_2,\dots,\mathbf w_\ell$ such that $\mathbf w_1=x^mt_1xt_2x\cdots t_nx$, $\mathbf w_\ell=x^{m+n}t_1t_2\cdots t_n$ and $\mathbf w_i\approx \mathbf w_{i+1}$ holds in either $\mathbf X$ or $\mathbf Y$ for any $i=1,2,\dots,\ell-1$. 
By symmetry, we may assume that $\mathbf w_1\approx \mathbf w_2$ is satisfied by $\mathbf X$.
Since $M(xy)\in \mathbf X$, Lemma~\ref{L: decompositions of u and v} implies that $\mathbf w_2=x^{e_0}\prod_{i=1}^n(t_ix^{e_i})$ for some $e_0,e_1,\dots,e_n\in\mathbb N_0$.
In view of Lemma~\ref{L: one letter in block reduction}, we may assume that $e_1,e_2,\dots,e_n\le 1$.
Now Lemma~\ref{L: reduction to delta_{r,m}} applies and we conclude that $\mathbf X$ satisfies $\delta_{n,m}$.
\end{proof}

\subsection{Identities formed by words with two multiple letters}

\begin{lemma}
\label{L: pxyq=pyxq in X wedge Y}
Let $\mathbf X$ and $\mathbf Y$ be aperiodic monoid subvarieties of $\mathbf A$. 
If $\mathbf X\wedge\mathbf Y$ satisfies the identity
\begin{equation}
\label{pxyq=pyxq}
\mathbf p\,xy\,\mathbf q\approx\mathbf p\,yx\,\mathbf q,
\end{equation}
where
\begin{equation}
\label{pxyq=pyxq restrict}
\begin{array}{l}
\mathbf p=a_1t_1\cdots a_kt_k\ \text{ and }\ \mathbf q=t_{k+1}a_{k+1}\cdots t_{k+\ell}a_{k+\ell}\ \text{ for some }\ k,\ell\in\mathbb N_0\\
\text{and}\ a_1,a_2,\dots,a_{k+\ell}\ \text{ are letters such that }\ \{a_1,a_2,\dots,a_{k+\ell}\}=\{x,y\},
\end{array}
\end{equation}
then this identity is true in either $\mathbf X$ or $\mathbf Y$.
\end{lemma}

\begin{proof}
If $M(xy)\notin \mathbf X$, then $\mathbf X$ is commutative by Lemma~\ref{L: does not contain M(xy)} and, therefore, satisfies~\eqref{pxyq=pyxq}.
By a similar argument we can show that $M(xy)\notin \mathbf Y$, then $\mathbf Y$ satisfies~\eqref{pxyq=pyxq}.
So, we may further assume that $M(xy)\in\mathbf X\wedge\mathbf Y$.

Put $\mathbf u=\mathbf p\,xy\,\mathbf q$. 
We note that $\simple(\mathbf u)=\{t_1,t_2,\dots,t_{k+\ell}\}$. 
If the word $\mathbf u_y$ is not an isoterm for $\mathbf X$, then $\mathbf X$ satisfies a non-trivial identity of the form $\mathbf u_y\approx\mathbf u^\prime$. 
Lemma~\ref{L: decompositions of u and v} implies that $\mathbf u^\prime=x^{f_0}\bigl(\prod_{i=1}^{k+\ell} t_ix^{f_i}\bigr)$ for some $f_0,f_1,\dots,f_{k+\ell}\in\mathbb N_0$.
In view of Lemma~\ref{L: one letter in block reduction}, we may assume that $f_1,f_2,\dots,f_{k+\ell}\le 1$.
Then we can apply Lemma~\ref{L: reduction to delta_{r,m}} with the conclusion that $\mathbf X$ satisfies the identity $\delta_{\occ_x(\mathbf u)-1,1}$. 
This identity together with $x^2y\approx yx^2$ implies $\mathbf u\approx x^{\occ_x(\mathbf u)}\,\mathbf u_x\approx\mathbf p\,yx\,\mathbf q$, and we are done. 
Analogous considerations show that if the word $\mathbf u_x$ is not an isoterm for $\mathbf X$ or $\mathbf Y$ or the word $\mathbf u_y$ is not an isoterm for $\mathbf Y$, then the identity~\eqref{pxyq=pyxq} is true in either $\mathbf X$ or $\mathbf Y$. 

So, it remains to consider the case both the words $\mathbf u_x$ and $\mathbf u_y$ are isoterms for $\mathbf X$ and $\mathbf Y$. 
Proposition~\ref{P: deduction} implies that one of the varieties $\mathbf X$ or $\mathbf Y$, say $\mathbf X$, satisfies a non-trivial identity of the form $\mathbf u\approx\mathbf v$. 
In view of Lemma~\ref{L: xt_1x...t_kx is isoterm}, the identity $\mathbf u\approx\mathbf v$ is linear-balanced. 
This means that $\simple(\mathbf v)=\{t_1,t_2,\dots,t_{k+\ell}\}$ and blocks of the word $\mathbf v$ (in order of their appearance from left to right) are $a_1$, $a_2$, \dots, $a_k$, $\mathbf w$, $a_{k+1}$, \dots, $a_{k+\ell}$, where $\mathbf w\in\{xy,yx\}$. 
Since the identity $\mathbf u\approx\mathbf v$ is non-trivial, $\mathbf w=yx$, whence $\mathbf v=\mathbf p\,yx\,\mathbf q$.
\end{proof}

\begin{lemma}
\label{L: xy.. = yx.. in X wedge Y}
Let $\mathbf X$ and $\mathbf Y$ be aperiodic monoid subvarieties of $\mathbf A\{\alpha,\,\beta\}$. 
If $\mathbf X\wedge\mathbf Y$ satisfies the identity
\begin{equation}
\label{two letters in a block}
x^{e_0}y^{f_0}\biggl(\prod_{i=1}^r t_ix^{e_i}y^{f_i}\biggr)\approx y^{f_0}x^{e_0}\biggl(\prod_{i=1}^r t_ix^{e_i}y^{f_i}\biggr),
\end{equation}
with $r\in\mathbb N_0$, $e_0,f_0\in\mathbb N$, $e_1,f_1,\dots,e_r,f_r\in\mathbb N_0$, $\sum_{i=0}^r e_i\ge 2$ and $\sum_{i=0}^r f_i\ge 2$, then this identity is true in either $\mathbf X$ or $\mathbf Y$.
\end{lemma}

\begin{proof}
If $e_i>1$ or $f_i>1$ for some $i\in\{0,1,\dots,r\}$, then it is routine to check that~\eqref{two letters in a block} is a consequence of $\{x^2y\approx yx^2,\,\beta\}$.
So, we may further assume that $e_0=f_0=1$ and $e_1,f_1,e_2,f_2,\dots,e_r,f_r\le 1$.

If $M(xy)\notin \mathbf X$, then $\mathbf X$ is commutative by Lemma~\ref{L: does not contain M(xy)} and, therefore, satisfies~\eqref{two letters in a block}.
By a similar argument we can show that $M(xy)\notin \mathbf Y$, then $\mathbf Y$ satisfies~\eqref{two letters in a block}.
So, we may further assume that $M(xy)\in\mathbf X\wedge\mathbf Y$.
Let $\mathbf u$ denote the the left-hand side of the identity~\eqref{two letters in a block}.
We note that $\simple(\mathbf u)=\{t_1,t_2,\dots,t_r\}$. 
If the word $\mathbf u_y$ is not an isoterm for $\mathbf X$, then $\mathbf X$ satisfies a non-trivial identity of the form $\mathbf u_y\approx\mathbf u^\prime$. 
Lemma~\ref{L: decompositions of u and v} implies that $\mathbf u^\prime=x^{g_0}\bigl(\prod_{i=1}^r t_ix^{g_i}\bigr)$ for some $g_0,g_1,\dots,g_r\in\mathbb N_0$.
In view of Lemma~\ref{L: one letter in block reduction}, we may assume that $g_1,g_2,\dots,g_r\le 1$.
Then we can apply Lemma~\ref{L: reduction to delta_{r,m}} with the conclusion that $\mathbf X$ satisfies the identity $\delta_{e-1,1}$, where $e=\sum_{i=0}^r e_i$. 
This identity together with $\{x^2y\approx yx^2,\,\beta\}$ implies 
$$
\mathbf u\stackrel{\delta_{e-1,1}}\approx x^e\,\mathbf u_x=x^ey^{f_0}\biggl(\prod_{i=1}^r t_iy^{f_i}\biggr)\stackrel{\{x^2y\approx yx^2,\,\beta\}}\approx y^{f_0}x^e\biggl(\prod_{i=1}^r t_iy^{f_i}\biggr) \stackrel{\delta_{e-1,1}}\approx y^{f_0}x^{e_0}\biggl(\prod_{i=1}^r t_ix^{e_i}y^{f_i}\biggr).
$$ 
Analogous considerations show that if the word $\mathbf u_x$ is not an isoterm for $\mathbf X$ or $\mathbf Y$ or the word $\mathbf u_y$ is not an isoterm for $\mathbf Y$, then the identity~\eqref{two letters in a block} is true in either $\mathbf X$ or $\mathbf Y$. 

So, it remains to consider the case when both the words $\mathbf u_x$ and $\mathbf u_y$ are isoterms for $\mathbf X$ and $\mathbf Y$. 
Since~\eqref{two letters in a block} is satisfied by in $\mathbf X\wedge\mathbf Y$, Proposition~\ref{P: deduction} implies that there is a sequence of pairwise distinct words $\mathbf w_1,\mathbf w_2,\dots,\mathbf w_\ell$ such that $\mathbf w_1=\mathbf u$, $\mathbf w_\ell$ is the right-hand side of~\eqref{two letters in a block} and $\mathbf w_i\approx \mathbf w_{i+1}$ holds in either $\mathbf X$ or $\mathbf Y$ for any $i=1,2,\dots,\ell-1$. 
Then there exists $j\in\{1,2,\dots,\ell-1\}$ such that $({_{1\mathbf w_j}x})<({_{1\mathbf w_j}y})$ but $({_{1\mathbf w_{j+1}}y})<({_{1\mathbf w_{j+1}}x})$.
Since $\mathbf u_x$ and $\mathbf u_y$ are isoterms for $\mathbf X\wedge \mathbf Y$, Lemma~\ref{L: xt_1x...t_kx is isoterm} implies that
$$
\mathbf w_j=xy\,\biggl(\prod_{i=1}^r t_i\mathbf a_i\biggr)\ \text{ and } \ \mathbf w_{j+1}=yx\,\biggl(\prod_{i=1}^r t_i\mathbf b_i\biggr),
$$
where $\mathbf a_i,\mathbf b_i\in\{x^{e_i}y^{f_i},y^{f_i}x^{e_i}\}$ for any $i=1,2,\dots,r$.
We may assume without any loss that $\mathbf w_j\approx \mathbf w_{j+1}$ holds in $\mathbf X$.
Then $\mathbf X$ satisfies~\eqref{two letters in a block} because
$$
xy\,\biggl(\prod_{i=1}^r t_ix^{e_i}y^{f_i}\biggr)\stackrel{\alpha}\approx xy\,\biggl(\prod_{i=1}^r t_i\mathbf a_i\biggr) \approx yx\,\biggl(\prod_{i=1}^r t_i\mathbf b_i\biggr)\stackrel{\alpha}\approx  yx\,\biggl(\prod_{i=1}^r t_ix^{e_i}y^{f_i}\biggr).
$$

Lemma~\ref{L: xy.. = yx.. in X wedge Y} is proved.
\end{proof}

\subsection{Identities of the form $\mathbf a_{n,m}[\rho] \approx \mathbf a_{n,m}^\prime[\rho]$}

Let 
$$
\mathbf A^\prime =\mathbf A\{\mathbf a_{n,m}[\rho] \approx \mathbf a_{n,m}^\prime[\rho]\mid (n,m)\in\hat{\mathbb N}_0^2\ \text{ and }\ \rho\in S_{n,m}\}.
$$

\begin{lemma}[{Gusev and Vernikov~\cite[Lemma~3.8]{Gusev-Vernikov-21}}]
\label{L: from pxqxr to px^2qr}
Let $\mathbf V$ be a subvariety of $\mathbf A^\prime$. 
If $\mathbf w=\mathbf px\mathbf qx\mathbf r$ and $\con(\mathbf q)\subseteq\mul(\mathbf w)$, then $\mathbf V$ satisfies the identity $\mathbf w\approx\mathbf px^2\mathbf q\mathbf r$.\qed
\end{lemma}

\begin{lemma}
\label{L: x^2 is a factor or occ_x > 2}
Let $\mathbf V$ be a subvariety of $\mathbf A$ satisfiying the identities~\eqref{xyzxy=yxzxy} and
\begin{equation}
\label{xyzxy=xyzyx}
xyzxy\approx xyzyx,
\end{equation}
$(n,m)\in\hat{\mathbb N}_0^2$ and $\rho\in S_{n,m}$.
Suppose that $\mathbf V$ satisfies a non-trivial identity $\mathbf a_{n,m}[\rho] \approx \mathbf a$ for some $\mathbf a\in\mathfrak X^\ast$ with $\mathbf a_x=\hat{\mathbf a}_{n,m}[\rho]$.
Suppose also that one of the following holds:
\begin{itemize}
\item[\textup{(i)}] $x^2$ is a factor of $\mathbf a$;
\item[\textup{(ii)}] $\occ_x(\mathbf a)>2$.
\end{itemize}
Then $\mathbf V$ satisfies the identity $\mathbf a_{n,m}[\rho] \approx \mathbf a_{n,m}^\prime[\rho]$.
\end{lemma}

\begin{proof}
Clearly,
$$
\mathbf a = \biggl(\prod_{i=1}^n x^{e_i}z_ix^{f_i}t_i\biggr)x^{g_0}\biggl(\prod_{i=1}^{n+m} z_{i\rho}x^{g_i}\biggr)\biggl(\prod_{i=n+1}^{n+m} t_ix^{e_i}z_ix^{f_i}\biggr)
$$
for some $g_0,e_1,f_1,g_1,\dots,e_{n+m},f_{n+m},g_{n+m}\in\mathbb N_0$.
Then $\mathbf V$ satisfies
\begin{equation}
\label{remove all z_i}
\biggl(\prod_{i=1}^n t_i\biggr)x^2\biggl(\prod_{i=n+1}^{n+m} t_i\biggr)\approx \biggl(\prod_{i=1}^n x^{e_i+f_i}t_i\biggr)x^{\sum_{i=0}^{n+m}{g_i}}\biggl(\prod_{i=n+1}^{n+m} t_ix^{e_i+f_i}\biggr).
\end{equation}

(i) If $x^2$ is a factor of $\mathbf a$, then $\mathbf V$ satisfies the identities
$$
\mathbf a\stackrel{\{x^2y\approx yx^2,\,\eqref{xyzxy=yxzxy},\,\eqref{xyzxy=xyzyx}\}}\approx \biggl(\prod_{i=1}^n x^{e_i+f_i}z_it_i\biggr)x^{\sum_{i=0}^{n+m}g_i}\biggl(\prod_{i=1}^{n+m} z_{i\rho}\biggr)\biggl(\prod_{i=n+1}^{n+m} t_ix^{e_i+f_i}z_i\biggr)\stackrel{\eqref{remove all z_i}}\approx \mathbf a_{n,m}^\prime[\rho],
$$ 
and we are done. 

\smallskip

(ii) In view of Part~(i), we may further assume that $x^2$ is not a factor of $\mathbf a$. 

First, we consider the case when $n+m=1$.
Let $\varepsilon$ denote the trivial permutation from $S_1$.
If $n=1$ and $m=0$, then $\mathbf a_{n,m}[\rho]\approx\mathbf a$ is nothing but $\mathbf a_{1,0}[\varepsilon]=z_1t_1xz_1x \approx x^{e_1}z_1x^{f_1}t_1x^{g_0}z_1x^{g_1}$.
Since $x^2$ is not a factor of $\mathbf a$ and $\occ_x(\mathbf a)>2$, at least three of the numbers $e_1$, $f_1$, $g_0$, $g_1$ are equal to~$1$.
Then $\mathbf V$ satisfies
$$
\mathbf a_{1,0}[\varepsilon]\approx \mathbf a= x^{e_1}z_1x^{f_1}t_1x^{g_0}z_1x^{g_1} \stackrel{\{\eqref{xyzxy=yxzxy},\,\eqref{xyzxy=xyzyx}\}}\approx x^{e_1+f_1}z_1t_1x^{g_0+g_1}z_1 \stackrel{\eqref{remove all z_i}}\approx z_1t_1x^2z_1=\mathbf a_{1,0}^\prime[\varepsilon],
$$
and we are done.
By a similar argument we can show that if $n=0$ and $m=1$, then $\mathbf a_{0,1}[\rho] \approx \mathbf a_{0,1}^\prime[\varepsilon]$ holds in $\mathbf V$.
So, we may further assume that $n,m\ge1$.

If every block of $\mathbf a$ contains at most one occurrence of $x$, then $\mathbf a \approx \mathbf a_{n,m}^\prime[\rho]$ is a consequence of~\eqref{remove all z_i}, whence $\mathbf V$ satisfies $\mathbf a_{n,m}[\rho] \approx \mathbf a_{n,m}^\prime[\rho]$.
So, it remains to consider the case when $x$ is multiple in some block of $\mathbf a$.

Suppose that $e_j+f_j>1$ for some $j\in\{1,2,\dots,n+m\}$.
Then $e_j=f_j=1$.
We may assume without loss of generality that $j\le n$.
Clearly, $x^2$ is a factor of $\mathbf a(x,z_{n+1},t_{n+1})$ and $\mathbf a_{n,m}[\rho](x,z_{n+1},t_{n+1})$ coincides (up to renaming of letters) with $\mathbf a_{0,1}[\varepsilon]$.
Then Part~(i) implies that $\mathbf V$ satisfies $\mathbf a_{0,1}[\varepsilon]\approx \mathbf a_{0,1}^\prime[\varepsilon]$.
Evidently,
$$
\mathbf a\approx \biggl(\prod_{i=1}^{j-1} x^{e_i}z_ix^{f_i}t_i\biggr) \cdot(x^{e_j+f_j}z_jt_j)\cdot \biggl(\prod_{i=j+1}^n x^{e_i}z_ix^{f_i}t_i\biggr)\biggl(\prod_{i=1}^{n+m} z_{i\rho}x^{g_i}\biggr)\!\biggl(\prod_{i=n+1}^{n+m} t_ix^{e_i}z_ix^{f_i}\biggr)
$$
follows from $\mathbf a_{0,1}[\varepsilon]\approx \mathbf a_{0,1}^\prime[\varepsilon]$.
Now Part~(i) applies again and we conclude that $\mathbf a_{n,m}[\rho]  \approx \mathbf a_{n,m}^\prime[\rho]$ is satisfied by $\mathbf V$.

Suppose now that $e_i+f_i\le 1$ for any $i=1,2,\dots,n+m$.
Then $\sum_{i=0}^{n+m}g_i\ge2$.
In this case, there are $1\le s\le r\le n+m$ such that $xz_{s\rho}z_{(s+1)\rho}\cdots z_{r\rho}x$ is a factor of $\mathbf a$.
Let 
$$
X=\{x,t_1,t_2,\dots,t_{n+m},z_{s\rho},z_{(s+1)\rho},\dots ,z_{r\rho}\}.
$$
It is easy to see that the identity $\mathbf a_{n,m}[\rho](X)\approx \mathbf a(X)$ implies the identity
\begin{equation}
\label{power of a}
\mathbf a \approx \biggl(\prod_{i=1}^n \mathbf e_it_i\biggr) \biggl(\prod_{i=1}^{s-1} x^{g_{i-1}}z_{i\rho}\biggr) x^{\sum_{i=0}^{s-1} g_i} \biggl(\prod_{i=s}^r z_{i\rho}\biggr) x^{\sum_{i=t}^{n+m} g_i}\biggl(\prod_{i=t+1}^{n+m} z_{i\rho}x^{g_i}\biggr)\biggl(\prod_{i=n+1}^{n+m} \mathbf e_it_i\biggr),
\end{equation}
where
$$ 
\mathbf e_i= 
\begin{cases} 
x^{2e_i}z_ix^{2f_i} & \text{if }\ s\le i\rho\le r;\\ 
x^{2e_i+f_i}z_ix^{f_i} & \text{if }\ e_i=1 \ \text{ and }\  i\rho<s\ \text{ or }\ r<i\rho;\\ 
x^{e_i}z_ix^{e_i+2f_i} & \text{if }\ e_i=0 \ \text{ and }\  i\rho<s\ \text{ or }\ r<i\rho.
\end{cases} 
$$
Evidently, if $e_j=1$ or $f_j=1$ for some $1\le j\le n+m$, then $x^2$ is a factor of the left hand-side of~\eqref{power of a}.
If $e_i=f_i=0$ for any $i=1,2,\dots,n+m$, then either $\sum_{i=0}^{s-1} g_i>1$ or $\sum_{i=r}^{n+m} g_i>1$ because $\occ_x(\mathbf a)>2$ and, therefore, $x^2$ is a factor of the left hand-side of~\eqref{power of a} as well.
We see that the left hand-side of~\eqref{power of a} contains the factor $x^2$ in either case.
According to Part~(i), the identity $\mathbf a_{n,m}[\rho] \approx \mathbf a_{n,m}^\prime[\rho]$ holds in $\mathbf V$. 

Lemma~\ref{L: x^2 is a factor or occ_x > 2} is proved.
\end{proof}

For any $n,m\in\mathbb N_0$, $\rho\in S_{n+m}$ and $0\le p\le q\le n+m$, we put
$$
\mathbf a_{n,m}^{p,q}[\rho]=\biggl(\prod_{i=1}^n z_it_i\biggr)\biggl(\prod_{i=1}^p z_{i\rho}\biggr)x\biggl(\prod_{i=p+1}^q z_{i\rho}\biggr)x\biggl(\prod_{i=q+1}^{n+m} z_{i\rho}\biggr)\biggl(\prod_{i=n+1}^{n+m} t_iz_i\biggr).
$$

\begin{lemma}
\label{L: satisfies a_{n+m}[rho]=a_n^{n+m}[rho]}
Let $\mathbf V$ be a monoid variety such that $\mathbf M(xyx)\subseteq\mathbf V\subseteq \mathbf A\{\eqref{xyzxy=yxzxy},\,\eqref{xyzxy=xyzyx}\}$.
If, for any $(n,m)\in\hat{\mathbb N}_0^2$ and $\rho\in S_{n,m}$, the monoid $M(\mathbf a_{n,m}[\rho])$ does not lie in $\mathbf V$, then $\mathbf V\subseteq\mathbf A^\prime$.
\end{lemma}

\begin{proof}
If $M(xzxyty)\notin \mathbf V$, then $\mathbf V$ satisfies $\beta$ by Lemma~\ref{L: V does not contain M(xyzxty) or M(xtyzxy) or M(xzxyty)}.
Then the same arguments as in the proof of Lemma~3.11 in~\cite{Gusev-Vernikov-21} imply that $\mathbf V$ satisfies $\mathbf a_{n,m}[\rho]\approx \mathbf a_{n,m}^\prime[\rho]$ for any $(n,m)\in\hat{\mathbb N}_0^2$ and $\rho\in S_{n,m}$.
So, we may assume that $M(xzxyty)\in \mathbf V$.

Since $M(\mathbf a_{1,0}[\varepsilon])\notin\mathbf V$, it follows from Lemma~\ref{L: M(W) in V} that $\mathbf V$ satisfies a non-trivial identity $\mathbf a_{1,0}[\varepsilon]\approx \mathbf a$ for some $\mathbf a\in \mathfrak X^\ast$. 
Lemma~\ref{L: isoterms for M(xzxyty)} implies that $\mathbf a_x=\hat{\mathbf a}_{1,0}[\varepsilon]$.
If $x^2$ is a factor of $\mathbf a$ or $\occ_x(\mathbf a)>2$, then $\mathbf V$ satisfies the identity $\mathbf a_{1,0}[\varepsilon] \approx \mathbf a_{1,0}^\prime[\varepsilon]$ by Lemma~\ref{L: x^2 is a factor or occ_x > 2}.
If $x^2$ is not a factor of $\mathbf a$ and $\occ_x(\mathbf a)\le 2$, then, since $xyx$ is an isoterm for $\mathbf V$, the identity $\mathbf a_{1,0}[\varepsilon]\approx \mathbf a$ must coincide (up to renaming of letters) with
\begin{equation}
\label{ztxzx=xzxtz}
ztxzx\approx xzxtz.
\end{equation}
By a similar argument one can show that the assumption that $M(\mathbf a_{0,1}[\varepsilon])\notin\mathbf V$ implies that $\mathbf V$ satisfies one of the identities $\mathbf a_{0,1}[\varepsilon]\approx \mathbf a_{0,1}^\prime[\varepsilon]$ or~\eqref{ztxzx=xzxtz}.

Let $\tau\in S_2=S_{1,1}$.
Since $M(\mathbf a_{1,1}[\tau])\notin\mathbf V$, Lemma~\ref{L: M(W) in V} implies that $\mathbf V$ satisfies a non-trivial identity $\mathbf a_{1,1}[\tau]\approx \mathbf a$ for some $\mathbf a\in \mathfrak X^\ast$. 
It follows from Lemma~\ref{L: isoterms for M(xzxyty)} that $\mathbf a_x=\hat{\mathbf a}_{1,1}[\tau]$.
If $x^2$ is a factor of $\mathbf a$ or $\occ_x(\mathbf a)>2$, then $\mathbf a_{1,1}[\tau] \approx \mathbf a_{1,1}^\prime[\tau]$ holds in $\mathbf V$ by Lemma~\ref{L: x^2 is a factor or occ_x > 2}.
Suppose now that $\occ_x(\mathbf a)\le2$ and $x^2$ is not a factor of $\mathbf a$.
Since $xyx$ and so $x$ are isoterms for $\mathbf V$, we have $\occ_x(\mathbf a)=2$.
Assume that some occurrence of $x$ lies between the first occurrences of $t_1$ and $t_2$ in $\mathbf a$.
Since $x^2$ is not a factor of $\mathbf a$ and the identity $\mathbf a_{1,1}[\tau]\approx \mathbf a$ is non-trivial, we may assume without any loss that $\mathbf a=\mathbf a_{1,1}^{1,2}[\tau]$.
Hence $(\mathbf a_{1,1}[\tau])_{\{z_{2\tau},t_{2\tau},x\}} \approx \mathbf a_{\{z_{2\tau},t_{2\tau},x\}}$ and $x^2y\approx yx^2$ imply either $\mathbf a_{1,0}[\varepsilon] \approx \mathbf a_{1,0}^\prime[\varepsilon]$ or $\mathbf a_{0,1}[\varepsilon] \approx \mathbf a_{0,1}^\prime[\varepsilon]$.
Then both $\mathbf a_{0,1}[\varepsilon] \approx \mathbf a_{0,1}^\prime[\varepsilon]$ and $\mathbf a_{1,0}[\varepsilon] \approx \mathbf a_{1,0}^\prime[\varepsilon]$ hold in $\mathbf V$ because
$$
\mathbf A\{\mathbf a_{1,0}[\varepsilon] \approx \mathbf a_{1,0}^\prime[\varepsilon],\,\eqref{ztxzx=xzxtz}\}=\mathbf A\{\mathbf a_{0,1}[\varepsilon] \approx \mathbf a_{0,1}^\prime[\varepsilon],\,\eqref{ztxzx=xzxtz}\}.
$$
Then $\mathbf V$ satisfies $\mathbf a_{1,1}[\tau] \approx \mathbf a_{1,1}^\prime[\tau]$ because $\mathbf a=\mathbf a_{1,1,}^{1,2}[\tau] \approx \mathbf a_{1,1}^\prime[\tau]$ is a consequence of 
$$
\{\mathbf a_{1,0}[\varepsilon] \approx \mathbf a_{1,0}^\prime[\varepsilon],\,\mathbf a_{0,1}[\varepsilon] \approx \mathbf a_{0,1}^\prime[\varepsilon],\,x^2y\approx yx^2\}.
$$
Assume now that there are no occurrences of $x$ between the first occurrences of $t_1$ and $t_2$ in $\mathbf a$. 
Then we may assume without any loss that some occurrence of $x$ precedes the first occurrence of $t_1$ in $\mathbf a$.
Since $xyx$ is an isoterm for $\mathbf V$ and $x^2$ is not a factor of $\mathbf a$, we have $\mathbf a=xz_1xt_1z_{1\tau}z_{2\tau}t_2z_2$.
Recall that either~\eqref{ztxzx=xzxtz} or $\mathbf a_{0,1}[\varepsilon]\approx  \mathbf a_{0,1}^\prime[\varepsilon]$ holds in $\mathbf V$.
If~\eqref{ztxzx=xzxtz} holds in $\mathbf V$, then $\mathbf V$ satisfies a non-trivial identity $\mathbf a\stackrel{\eqref{ztxzx=xzxtz}}\approx\mathbf a_{1,1}^{\ell_1,\ell_2}[\tau]$ for some $0\le \ell_1\le \ell_2\le 2$.
In this case, the same arguments as in the above show that $\mathbf a_{1,1}[\tau]\approx\mathbf a_{1,1}^\prime[\tau]$ holds in $\mathbf V$.
If $\mathbf a_{0,1}[\varepsilon]\approx  \mathbf a_{0,1}^\prime[\varepsilon]$ holds in $\mathbf V$, then $\mathbf V$ satisfies $\mathbf a\stackrel{\mathbf a_{0,1}[\varepsilon]\approx  \mathbf a_{0,1}^\prime[\varepsilon]}\approx x^2z_1t_1z_{1\tau}z_{2\tau}t_2z_2 \stackrel{x^2y\approx yx^2}\approx \mathbf a_{1,1}^\prime[\tau]$.
We see that $\mathbf a_{1,1}[\tau]\approx \mathbf a_{1,1}^\prime[\tau]$ holds in $\mathbf V$ in either case.
Then $\mathbf V$ satisfies also the identities $\mathbf a_{0,1}[\varepsilon] \approx \mathbf a_{0,1}^\prime[\varepsilon]$ and $\mathbf a_{1,0}[\varepsilon] \approx \mathbf a_{1,0}^\prime[\varepsilon]$ because they are consequences of $\mathbf a_{1,1}[\tau] \approx \mathbf a_{1,1}^\prime[\tau]$.

We have proved that there exists a number $r$ such that $\mathbf V$ satisfies the identity $\mathbf a_{r_1,r_2}[\rho]\approx \mathbf a_{r_1,r_2}^\prime[\rho]$ for all $(r_1,r_2)\in\hat{\mathbb N}_0^2$ and $\rho\in S_{r_1,r_2}$ with $r_1+r_2\le r$ (for instance, $r=1$). 
We are going to verify that an arbitrary $r$ possesses this property. 
Arguing by contradiction, we suppose that the mentioned claim is true for $r=1,2,\dots,k-1$ but is false for $r=k$.
Then $\mathbf V$ violates $\mathbf a_{n,m}[\rho]\approx\mathbf a_{n,m}^\prime[\rho]$ for some $(n,m)\in\hat{\mathbb N}_0^2$ and $\rho\in S_{n,m}$ such that $n+m=k$.
According to Lemma~\ref{L: M(W) in V}, $\mathbf V$ satisfies a non-trivial identity $\mathbf a_{n,m}[\rho] \approx \mathbf a$.
In view of Lemma~\ref{L: isoterms for M(xzxyty)}, $\mathbf a_x=\hat{\mathbf a}_{n,m}[\rho]$.
Then $\occ_x(\mathbf a)\le2$ and $x^2$ is not a factor of $\mathbf a$ by Lemma~\ref{L: x^2 is a factor or occ_x > 2}.
It follows from the fact that $x$ is an isoterm for $\mathbf V$ that $\occ_x(\mathbf a)=2$.

Suppose that some occurrence of $x$ lies between the first occurrences of $t_n$ and $t_{n+1}$ in $\mathbf a$.
Since $x^2$ is not a factor of $\mathbf a$ and the identity $\mathbf a_{n,m}[\rho]\approx \mathbf a$ is non-trivial, we may assume without any loss that $\mathbf a=\mathbf a_{n,m}^{p,q}[\rho]$ for some $0<p< q\le n+m$.
Let 
$$
X=\{z_{1\rho},t_{1\rho},z_{2\rho},t_{2\rho},\dots,z_{p\rho},t_{p\rho},z_{(q+1)\rho},t_{(q+1)\rho},z_{(q+2)\rho},t_{(q+2)\rho},\dots,z_{k\rho},t_{k\rho}\}.
$$
Clearly, $(\mathbf a_{n,m}[\rho])_X$ coincides (up to renaming of letters) with $\mathbf a_{c,d}[\pi]$ for some $(c,d)\in\hat{\mathbb N}_0^2$ and $\pi \in S_{c,d}$ such that $c+d=q-p$.
Since $\mathbf a_{c,d}[\pi]\approx \mathbf a_{c,d}^\prime[\pi]$ holds in the variety $\mathbf V$, this variety must satisfy $\mathbf a=\mathbf a_{n,m}^{p,q}[\rho]  \stackrel{x^2y\approx yx^2}\approx \mathbf a_{n,m}^{p,p}[\rho] \stackrel{x^2y\approx yx^2}\approx \mathbf a_{n,m}^\prime[\rho]$ contradicting the choice of $n,m$ and $\rho$.

Suppose now that there are no occurrences of $x$ between the first occurrences of $t_n$ and $t_{n+1}$ in $\mathbf a$. 
Then we may assume without any loss that some occurrence of $x$ precedes the first occurrence of $t_n$ in $\mathbf a$.
Since $xyx$ is an isoterm for $\mathbf V$ and $x^2$ is not a factor of $\mathbf a$, we have 
$$
\mathbf a=\biggl(\prod_{i=1}^{j-1} z_it_i\biggr)\cdot (xz_jxt_j)\cdot \biggl(\prod_{i=j+1}^n z_it_i\biggr)\biggl(\prod_{i=1}^{n+m} z_{i\rho}\biggr)\biggl(\prod_{i=n+1}^{n+m} t_iz_i\biggr)
$$
for some $j\in\{1,2,\dots,n\}$.
Taking into account the fact that the identities $\mathbf a_{0,1}[\varepsilon]\approx  \mathbf a_{0,1}^\prime[\varepsilon]$ and $x^2y\approx yx^2$ hold in $\mathbf V$, we obtain that $\mathbf V$ satisfies $\mathbf a_{n,m}[\rho]\approx \mathbf a_{n,m}^{j-1,j-1}[\rho]\approx  \mathbf a_{n,m}^\prime[\rho]$ contradicting the choice of $n,m$ and $\rho$ again.
\end{proof}

\subsection{Identities of the form $\mathbf c_{n,m,k}[\rho] \approx \mathbf c_{n,m,k}^\prime[\rho]$}

\begin{lemma}
\label{L: V notin M(c_{n,m,k}[rho])}
Let $\mathbf V$ be a variety such that $\mathbf M(xyx)\subseteq\mathbf V$ and $\mathbf N\nsubseteq\mathbf V$.
Suppose that $\mathbf V$ does not contain the monoids $M(\mathbf c_{n,m,0}[\tau])$ and $M(\mathbf c_{n,m,n+m+1}[\pi])$ for all $n,m\in\mathbb N_0$, $\tau\in S_{n+m}$ and $\pi\in S_{n+m,n+m+1}$. 
Then $\mathbf V$ satisfies the identity $\mathbf c_{n,m,k}[\rho] \approx \mathbf c_{n,m,k}[\rho]$ for any $n,m,k\in\mathbb N_0$ and $\rho\in S_{n+m+k}$.
\end{lemma}

\begin{proof}
It is routine to check that, for any $p,q,r\in\mathbb N_0$ and $\rho\in S_{p+q+r}$, one can construct $\pi\in S_{n+m,n+m+1}$ for some $n,m\in \mathbb N_0$ such that $\mathbf c_{p,q,r}[\rho] \approx \mathbf c_{p,q,r}[\rho]$ is a consequence of
\begin{equation}
\label{c_{n,m,n+m+1}[pi]=c_{n,m,n+m+1}[pi]}
\mathbf c_{n,m,n+m+1}[\pi] \approx \mathbf c_{n,m,n+m+1}^\prime[\pi].
\end{equation}
In view of this fact, it suffices to show that $\mathbf V$ satisfies the identity~\eqref{c_{n,m,n+m+1}[pi]=c_{n,m,n+m+1}[pi]} for any $n,m\in \mathbb N_0$ and $\pi\in S_{n+m,n+m+1}$. 

Let us fix $n,m\in \mathbb N_0$ and $\pi\in S_{n+m,n+m+1}$.
Put $k=n+m$.
Suppose that $xzxyty$ is an isoterm for $\mathbf V$.
Then it is easy to see that if $\mathbf c_{n,m,k+1}[\pi] \approx \mathbf c$ is an identity of $\mathbf V$, then 
$
\mathbf c\in\{\mathbf c_{n,m,k+1}[\pi],\mathbf c_{n,m,k+1}^\prime[\pi]\}
$.
Since $M(\mathbf c_{n,m,k+1}[\pi])\notin\mathbf V$, this fact and Lemma~\ref{L: M(W) in V} imply that~\eqref{c_{n,m,n+m+1}[pi]=c_{n,m,n+m+1}[pi]} holds in $\mathbf V$.
Thus, we may further assume that $xzxyty$ is not an isoterm for $\mathbf V$.
Then $\mathbf V$ satisfies $\beta$ by Lemma~\ref{L: V does not contain M(xyzxty) or M(xtyzxy) or M(xzxyty)}(ii).
It follows that $\mathbf V$ satisfies the identities
\begin{equation}
\label{two identities}
\begin{aligned}
\mathbf c_{n,m,k+1}[\pi] \stackrel{\beta}\approx\biggl(\prod_{i=1}^n z_it_i\biggr)xyt\biggl(\prod_{i=n+1}^k z_it_i\biggr)x\biggl(\prod_{i=1}^{2k+1} z_{i\pi}\biggr)_{Z_1}y\biggl(\prod_{i=1}^{2k+1} z_{i\pi}\biggr)_{Z_2}\biggl(\prod_{i=k+1}^{2k+1} t_iz_i\biggr),\\
\mathbf c_{n,m,k+1}^\prime[\pi] \stackrel{\beta}\approx\biggl(\prod_{i=1}^n z_it_i\biggr)yxt\biggl(\prod_{i=n+1}^k z_it_i\biggr)x\biggl(\prod_{i=1}^{2k+1} z_{i\pi}\biggr)_{Z_1}y\biggl(\prod_{i=1}^{2k+1} z_{i\pi}\biggr)_{Z_2}\biggl(\prod_{i=k+1}^{2k+1} t_iz_i\biggr),
\end{aligned}
\end{equation}
where $Z_1=\{z_i\mid k+1\le i\le 2k+1\}$ and $Z_2=\{z_i\mid 1\le i\le k\}$.
Evidently, there exists $\tau\in S_k$ such that 
$$
\mathbf c_{n,m,k}[\pi](X)=\mathbf c_{n,m,0}[\tau]\ \text{ and }\ \mathbf c_{n,m,k}^\prime[\pi](X)=\mathbf c_{n,m,0}^\prime[\tau],
$$
where $X=\{x,y,z_i,t_i\mid 1\le i \le k\}$.
Since $M(\mathbf c_{n,m,0}[\tau])\notin\mathbf V$, Lemma~3.14 in~\cite{Gusev-Vernikov-21} implies the $\mathbf V$ satisfies $\mathbf c_{n,m,0}[\tau]\approx\mathbf c_{n,m,0}^\prime[\tau]$.
Then the identity~\eqref{c_{n,m,n+m+1}[pi]=c_{n,m,n+m+1}[pi]} holds in $\mathbf V$ because it is a consequence of $\mathbf c_{n,m,0}[\tau]\approx\mathbf c_{n,m,0}^\prime[\tau]$ and~\eqref{two identities}. 
\end{proof}

\section{Proof of Theorem~\ref{T: A_cen}}
\label{Sec: proof}

\textit{Necessity}.
Let $\mathbf V$ be a distributive subvariety of $\mathbf A_\mathsf{cen}$.
In view of Lemma~\ref{L: subvariety of A_cen}, any subvariety of $\mathbf A_\mathsf{cen}$ satisfies the identities
\begin{equation}
\label{x^n=x^{n+1} and x^ny=yx^n}
x^n \approx x^{n+1}\ \text{ and } \ x^ny \approx yx^n
\end{equation}
for some $n\in \mathbb N$.
Let $n$ be the least number such that $\mathbf V$ satisfies~\eqref{x^n=x^{n+1} and x^ny=yx^n}.
If $n=1$, then $\mathbf V\subseteq\mathbf{SL}=\mathbf P_1$, and we are done.
So, we may further assume that $n>1$.
Two cases are possible.

\smallskip

\textit{Case }1: $\mathbf M(xyx)\subseteq\mathbf V$.
If $n=2$, then $\mathbf V$ satisfies the identity $x^2y \approx yx^2$ because this identity is nothing but the second identity in~\eqref{x^n=x^{n+1} and x^ny=yx^n}.
Let now $n>2$.
In view of Proposition~\ref{P: M(x^2y,yx^2)}, $M(x^2y)\notin\mathbf V$.
Then Lemma~\ref{L: M(W) in V} implies that $\mathbf V$ satisfies a non-trivial identity $x^2y\approx \mathbf v$.
According to Lemma~\ref{L: x^n is an isoterm}, $x^2$ is an isoterm for $\mathbf V$, whence $\con(\mathbf v)=\{x,y\}$, $\occ_x(\mathbf v)=2$ and $\occ_y(\mathbf v)=1$.
This is only possible when $\mathbf v=yx^2$ because $\mathbf M(xyx)\subseteq\mathbf V$ and the identity $x^2y\approx \mathbf v$ is non-trivial.
We see that $\mathbf V$ satisfies $x^2y\approx yx^2$ in either case.
Two subcases are possible.

\smallskip

\textit{Subcase }1.1: $\mathbf N\subseteq\mathbf V$ or $\mathbf N^\delta\subseteq\mathbf V$.
By symmetry, we may assume that $\mathbf N\subseteq\mathbf V$.
It is shown in~\cite[Theorem~1.1]{Gusev-19} that the lattice $\mathfrak L(\mathbf M(xtyzxy)\vee\mathbf N)$ is not modular.
Therefore, $M(xtyzxy)\notin \mathbf V$.
Then $\mathbf V$ satisfies $\alpha$ by Lemma~\ref{L: V does not contain M(xyzxty) or M(xtyzxy) or M(xzxyty)}(i).
In view of Lemma~\ref{P: M(xzxyty) vee N}, $M(xyzxty)\notin \mathbf V$.
Then Lemma~\ref{L: V does not contain M(xyzxty) or M(xtyzxy) or M(xzxyty)}(ii) implies that $\mathbf V$ satisfies $\beta$.
Hence $\mathbf V\subseteq \mathbf R_n$.

\smallskip

\textit{Subcase }1.2: $\mathbf N,\mathbf N^\delta\nsubseteq\mathbf V$.
According to Propositions~\ref{P: L(M(c_{n,m,n+m+1}[rho])) is not modular} and~\ref{P: L(M(c_{n,m,0}[rho])) is not modular} and the statements dual to they, the variety $\mathbf V$ does not contain the monoids 
$$
M(\mathbf c_{n,m,n+m+1}[\pi]),\ M(\mathbf d_{n,m,n+m+1}[\pi]),\ M(\mathbf c_{n,m,0}[\tau])\ \text{ and }\ M(\mathbf d_{n,m,0}[\tau])
$$
for all $n,m\in\mathbb N_0$, $\pi\in S_{n+m,n+m+1}$ and $\tau\in S_{n+m}$.
Then Lemma~\ref{L: V notin M(c_{n,m,k}[rho])} and the dual to it imply that $\mathbf V$ satisfies the identities
$$
\mathbf c_{n,m,k}[\rho]\approx \mathbf c_{n,m,k}^\prime[\rho]\ \text{ and }\ \mathbf d_{n,m,k}[\rho]\approx \mathbf d_{n,m,k}^\prime[\rho]
$$
for any $n,m,k\in\mathbb N_0$ and $\rho\in S_{n+m+k}$.
 
Further, $M(\mathbf a_{n,m}[\rho])\notin\mathbf V$ for any $(n,m)\in\mathbb N_0^2$ and $\rho\in S_{n,m}$ by Proposition~\ref{P: L(M(a_{n,m}[rho])) is not distributive}.
Then Lemma~\ref{L: satisfies a_{n+m}[rho]=a_n^{n+m}[rho]} implies that $\mathbf V$ satisfies $\mathbf a_{n,m}[\rho]\approx \mathbf a_{n,m}^\prime[\rho]$ for any $(n,m)\in\hat{\mathbb N}_0^2$ and $\rho\in S_{n,m}$.
It follows from the proof of Lemma~4.4 in~\cite{Gusev-Vernikov-18} that $\mathbf a_{n,m}[\rho]\approx \mathbf a_{n,m}^\prime[\rho]$ holds in $\mathbf V$ for any $n,m\in\mathbb N_0$ and $\rho\in S_{n+m}$.
Hence $\mathbf V\subseteq\mathbf P_n$.

\smallskip

\textit{Case }2: $\mathbf M(xyx)\nsubseteq\mathbf V$.
According to Lemma~\ref{L: M(W) in V}, $\mathbf V$ satisfies a non-trivial identity $xyx\approx \mathbf v$.
Since $n>1$, Lemma~\ref{L: x^n is an isoterm} implies that $x$ is an isoterm for $\mathbf V$.
Then $\mathbf v=x^syx^t$, where either $s\ge 2$ or $t\ge 2$.
By symmetry, we may assume that $s\ge 2$.
If $n=2$, then $\mathbf V$ satisfies $x^2y\approx xyx$ because
$$
xyx\approx x^syx^t \stackrel{x^2\approx x^3}\approx x^2yx^t \stackrel{x^2y\approx yx^2}\approx yx^{2+t}\stackrel{x^2\approx x^3}\approx yx^2 \stackrel{x^2y\approx yx^2}\approx x^2y.
$$
If $n>2$, then $s=2$ and $t=0$ by Lemma~\ref{L: x^n is an isoterm}.
We see that $x^2y\approx xyx$ is satisfied by $\mathbf V$ in either case, whence $\mathbf V\subseteq\mathbf Q_n$. 

\medskip

\textit{Sufficiency}.
By symmetry, it suffices to show that the varieties $\mathbf P_n$, $\mathbf Q_n$ and $\mathbf R_n$ are distributive.
It is shown in~\cite[Corollary~5.8]{Gusev-Vernikov-21} that the lattice $\mathfrak L(\mathbf Q_n)$ is distributive.
So, it remains to prove that the lattices $\mathfrak L(\mathbf P_n)$ and $\mathfrak L(\mathbf R_n)$ are distributive.

\smallskip

\textit{Distributivity of $\mathfrak L(\mathbf P_n)$}.
We are going to deduce the required fact from Lemma~\ref{L: smth imply distributivity} with $\mathbf V=\mathbf P_n$ and $\Sigma=\Phi$, where $\Phi$ consists of the identity $xy\approx yx$, all identities of the form $x^k \approx x^\ell$ with $k,\ell\in\mathbb N$, all identities of the form $\delta_{k,\ell}$ with $k,\ell\in\mathbb N_0$ and all identities of the form~\eqref{pxyq=pyxq} such that the equalities~\eqref{pxyq=pyxq restrict} hold. 
In view of Lemmas~\ref{L: x^n=x^m in X wedge Y} and~\ref{L: pxyq=pyxq in X wedge Y} and Corollaries~\ref{C: xy=yx in X wedge Y} and~\ref{C: delta_{n,m} in X wedge Y}, to do this, it remains to prove only that each subvariety of $\mathbf P_n$ may be given within $\mathbf P_n$ by some subset of $\Phi$.

Let $\mathbf w^\prime \approx \mathbf v^\prime$ be an arbitrary identity.
It suffices to verify that $\mathbf P_n\{\mathbf w^\prime \approx \mathbf v^\prime\}=\mathbf P_n\Gamma$ for some $\Gamma\subseteq\Phi$.
If $\mathbf P_n\{\mathbf w^\prime \approx \mathbf v^\prime\}$ is commutative, then it can be defined by some subset of $\{x^k\approx x^\ell,\,xy\approx yx\mid k,\ell\in\mathbb N\}$ by Proposition~\ref{P: commutative}.
Let us now consider the case when $\mathbf P_n\{\mathbf w^\prime \approx \mathbf v^\prime\}$ is not commutative.
Then $M(xy)\in \mathbf P_n\{\mathbf w^\prime \approx \mathbf v^\prime\}$ by Lemma~\ref{L: does not contain M(xy)}.
Let $\mathbf w_0^\prime\prod_{i=1}^m(t_i\mathbf w_i^\prime)$ be the decomposition of $\mathbf w^\prime$. 
Lemma~\ref{L: decompositions of u and v} implies that the decomposition of $\mathbf v^\prime$ has the form $\mathbf v_0^\prime\prod_{i=1}^m(t_i\mathbf v_i^\prime)$.
According to Lemma~\ref{L: from pxqxr to px^2qr}, the identities satisfied by $\mathbf P_n$ can be used to convert the words $\mathbf w^\prime$ and $\mathbf v^\prime$ into some words $\mathbf w$ and $\mathbf v$, respectively, such that the following hold:
\begin{itemize}
\item the decompositions of $\mathbf w$ and $\mathbf v$ are of the form $\mathbf w_0\prod_{i=1}^m(t_i\mathbf w_i)$ and $\mathbf v_0\prod_{i=1}^m(t_i\mathbf v_i)$, respectively;
\item $\occ_x(\mathbf w_i)=\occ_x(\mathbf w_i^\prime)$ and $\occ_x(\mathbf v_i)=\occ_x(\mathbf v_i^\prime)$ for any $x\in\mathfrak X$ and $i=0,1,\dots,m$;
\item $x^{\occ_x(\mathbf w_i)}$ and $x^{\occ_x(\mathbf v_i)}$ are factors of $\mathbf w_i$ and $\mathbf v_i$, respectively, for any $x\in\con(\mathbf w_i)=\con(\mathbf v_i)$ and $i=0,1,\dots,m$.
\end{itemize}
In view of this fact, it suffices to show that the identity $\mathbf w \approx \mathbf v$ is equivalent within $\mathbf P_n$ to some subset of $\Phi$.
Clearly, the set 
\begin{equation}
\label{set of identities}
\{\mathbf w(x,t_1,t_2,\dots,t_m)\approx \mathbf v(x,t_1,t_2,\dots,t_m)\mid x\in \mul(\mathbf w)=\mul(\mathbf v)\}
\end{equation}
of identities can be used to convert the word $\mathbf w$ into some word $\mathbf u$ such that the following hold:
\begin{itemize}
\item the decomposition $\mathbf u$ has the form $\mathbf u_0\prod_{i=1}^m(t_i\mathbf u_i)$;
\item $\occ_x(\mathbf u_i)=\occ_x(\mathbf v_i)$ for any $x\in\mathfrak X$ and $i=1,2,\dots,m$;
\item $x^{\occ_x(\mathbf u_i)}$ is a factor of $\mathbf u_i$ for any $x\in\con(\mathbf u_i)$ and $i=1,2,\dots,m$.
\end{itemize}
Evidently, every identity from~\eqref{set of identities} is of the form~\eqref{one letter in a block}.
Then it follows from Lemmas~\ref{L: one letter in block reduction} and~\ref{L: reduction to delta_{r,m}} that the set~\eqref{set of identities} is equivalent modulo $x^2y\approx yx^2$ to some identities of the form $\delta_{k,\ell}$ and $x^k\approx x^\ell$.
In view of this fact, it remains to show that the identity $\mathbf u \approx \mathbf v$ is equivalent within $\mathbf P_n$ to a subset of $\Phi$.

We call an identity $\mathbf c\approx\mathbf d$ 1-\textit{invertible} if $\mathbf c=\mathbf e^\prime\, xy\,\mathbf e^{\prime\prime}$ and $\mathbf d=\mathbf e^\prime\, yx\,\mathbf e^{\prime\prime}$ for some words $\mathbf e^\prime,\mathbf e^{\prime\prime}$ and letters $x,y\in\con(\mathbf e^\prime\mathbf e^{\prime\prime})$. 
Let $k>1$. 
An identity $\mathbf c\approx\mathbf d$ is called $k$-\textit{invertible} if there is a sequence of words $\mathbf c=\mathbf w_0,\mathbf w_1,\dots,\mathbf w_k=\mathbf d$ such that the identity $\mathbf w_i\approx\mathbf w_{i+1}$ is 1-invertible for each $i=0,1,\dots,k-1$ and $k$ is the least number with such a property. 
For convenience, we will call the trivial identity 0-\textit{invertible}. 

Notice that the identity $\mathbf u \approx \mathbf v$ is $r$-invertible for some $r\in\mathbb N_0$ because $\occ_x(\mathbf u_i)=\occ_x(\mathbf v_i)$ for any $x\in\mathfrak X$ and $i=1,2,\dots,m$. 
We will use induction by $r$.

\smallskip

\textit{Induction base}. 
If $r=0$, then $\mathbf u=\mathbf v$, whence $\mathbf P_n\{\mathbf u\approx\mathbf v\}=\mathbf P_n\{\varnothing\}$.

\smallskip

\textit{Induction step}. 
Let $r>0$. 
Obviously, $\mathbf u_j\ne\mathbf v_j$ for some $j\in\{0,1,\dots,m\}$. 
Then there are letters $x$ and $y$ such that $\mathbf u_j=\mathbf b^\prime\, y^qx^p\,\mathbf b^{\prime\prime}$ for some $\mathbf b^\prime,\mathbf b^{\prime\prime}\in\mathfrak X^\ast$ and the word $x^p$ precedes the word $y^q$ in the block $\mathbf v_j$, where $p=\occ_x(\mathbf u_j)=\occ_x(\mathbf v_j)$ and $q=\occ_y(\mathbf u_j)=\occ_y(\mathbf v_j)$. 
We denote by $\hat{\mathbf u}$ the word that is obtained from $\mathbf u$ by swapping of the words $x^p$ and $y^q$ in the block $\mathbf u_j$. 

Suppose that $\occ_x(\mathbf u_c)>1$ for some $c\in\{0,1,\dots,m\}$. 
We may assume without loss of generality that $j\le c$.
If $c\ne j$, then $\mathbf u_c=\mathbf f^\prime\,x^{\occ_x(\mathbf u_c)}\,\mathbf f^{\prime\prime}$ for some $\mathbf f^\prime,\mathbf f^{\prime\prime}\in\mathfrak X^\ast$ and $\mathbf P_n$ satisfies the identities
$$
\begin{aligned}
\mathbf u&{}\stackrel{\,\,\,\,\,\,x^2y\approx yx^2\,\,\,\,\,}\approx \mathbf u_0\cdot\biggl(\prod_{i=1}^{j-1}t_i\mathbf u_i\biggr)\cdot(t_j\mathbf b^\prime\, x^2y^qx^p\,\mathbf b^{\prime\prime})\cdot\biggl(\prod_{i=j+1}^{c-1}t_i\mathbf u_i\biggr)\cdot(t_c\mathbf f^\prime\,x^{\occ_x(\mathbf u_c)-2}\,\mathbf f^{\prime\prime})\cdot\biggl(\prod_{i=c+1}^mt_i\mathbf u_i\biggr)\\
&{}\stackrel{\text{Lemma}~\ref{L: from pxqxr to px^2qr}}\approx \mathbf u_0\cdot\biggl(\prod_{i=1}^{j-1}t_i\mathbf u_i\biggr)\cdot(t_j\mathbf b^\prime\, x^{2+p}y^q\,\mathbf b^{\prime\prime})\cdot\biggl(\prod_{i=j+1}^{c-1}t_i\mathbf u_i\biggr)\cdot(t_c\mathbf f^\prime\,x^{\occ_x(\mathbf u_c)-2}\,\mathbf f^{\prime\prime})\cdot\biggl(\prod_{i=c+1}^mt_i\mathbf u_i\biggr)\\
&{}\stackrel{\,\,\,\,\,\,x^2y\approx yx^2\,\,\,\,\,}\approx \mathbf u_0\cdot\biggl(\prod_{i=1}^{j-1}t_i\mathbf u_i\biggr)\cdot(t_j\mathbf b^\prime\, x^py^q\,\mathbf b^{\prime\prime})\cdot\biggl(\prod_{i=j+1}^{c-1}t_i\mathbf u_i\biggr)\cdot(t_c\mathbf f^\prime\,x^{\occ_x(\mathbf u_c)}\,\mathbf f^{\prime\prime})\cdot\biggl(\prod_{i=c+1}^mt_i\mathbf u_i\biggr)=\hat{\mathbf u}.
\end{aligned}
$$
If $c=j$, then $\mathbf P_n$ satisfies the identities
$$
\begin{aligned}
\mathbf u&{}\stackrel{\,\,\,\,\,\,x^2y\approx yx^2\,\,\,\,\,}\approx \mathbf u_0\cdot\biggl(\prod_{i=1}^{j-1}t_i\mathbf u_i\biggr)\cdot(t_j\mathbf b^\prime\, x^2y^qx^{p-2}\mathbf b^{\prime\prime})\cdot\biggl(\prod_{i=j+1}^mt_i\mathbf u_i\biggr)\\
&{}\stackrel{\text{Lemma}~\ref{L: from pxqxr to px^2qr}}\approx \mathbf u_0\cdot\biggl(\prod_{i=1}^{j-1}t_i\mathbf u_i\biggr)\cdot(t_j\mathbf b^\prime\, x^py^q\,\mathbf b^{\prime\prime})\cdot\biggl(\prod_{i=j+1}^mt_i\mathbf u_i\biggr)=\hat{\mathbf u}.
\end{aligned}
$$
We see that $\hat{\mathbf u}\approx\mathbf v$ is satisfied by $\mathbf P_n$.
The identity $\hat{\mathbf u}\approx\mathbf v$ is \mbox{$(r-pq)$}-invertible. 
By the induction assumption, $\mathbf P_n\{\hat{\mathbf u}\approx\mathbf v\}=\mathbf P_n\Gamma$ for some $\Gamma\subseteq\Phi$. 
Then $\mathbf P_n\{\mathbf u\approx\mathbf v\}=\mathbf P_n\Gamma$, and we are done.
By similar arguments we can show that if $\occ_y(\mathbf u_d)=\occ_y(\mathbf v_d)>1$ for some $d\in\{0,1,\dots,m\}$, then $\mathbf P_n\{\mathbf u\approx\mathbf v\}=\mathbf P_n\Gamma$ for some $\Gamma\subseteq\Phi$. 
Thus, we may further assume that $\occ_x(\mathbf u_i),\occ_y(\mathbf u_i)\le 1$ for any $i=0,1,\dots,m$.
In particular, $p=q=1$.

Suppose that $x,y\in\con(\mathbf u_a)=\con(\mathbf v_a)$ for some $a\ne j$. 
Then Lemma~3.12 in~\cite{Gusev-Vernikov-21} or the statement dual to it implies that $\mathbf P_n$ satisfies the identity $\mathbf u\approx\hat{\mathbf u}$. 
The identity $\hat{\mathbf u}\approx\mathbf v$ is \mbox{$(r-1)$}-invertible. 
By the induction assumption, $\mathbf P_n\{\hat{\mathbf u}\approx\mathbf v\}=\mathbf P_n\Gamma$ for some $\Gamma\subseteq\Phi$. 
Then $\mathbf P_n\{\mathbf u\approx\mathbf v\}=\mathbf P_n\Gamma$, and we are done.
Thus, we may further assume that at most one of the letters $x$ and $y$ occurs in $\mathbf u_i$ for any $i\ne j$.

Let $s$ be the least number such that $\con(\mathbf u_s)\cap\{x,y\}\ne\varnothing$. Put
\begin{equation}
\label{T_{x,y}}
T_{x,y}=\{t_j\mid s<j\le m,\ \con(\mathbf u_j)\cap\{x,y\}\ne\varnothing\}.
\end{equation}
In view of the above, the identity
\begin{equation}
\label{2 mult letters and their dividers}
\mathbf u(\{x,y\}\cup T_{x,y})\approx\mathbf v(\{x,y\}\cup T_{x,y})
\end{equation}
must be linear-balanced.
Then, since at most one of the letters $x$ and $y$ occurs in $\mathbf u_i$ for any $i\ne j$, the identity~\eqref{2 mult letters and their dividers} must coincide (up to renaming of letters) with an identity of the form~\eqref{pxyq=pyxq} such that the equalities~\eqref{pxyq=pyxq restrict} hold. 
It is clear that $\mathbf u\stackrel{\eqref{2 mult letters and their dividers}}\approx\hat{\mathbf u}$ and the identity $\hat{\mathbf u}\approx\mathbf v$ is \mbox{$(r-1)$}-invertible. 
By the induction assumption, $\mathbf P_n\{\hat{\mathbf u}\approx\mathbf v\}=\mathbf P_n\Gamma$ for some $\Gamma\subseteq\Phi$. 
Then $\mathbf P_n\{\mathbf u\approx\mathbf v\}=\mathbf P_n\{\eqref{2 mult letters and their dividers},\,\Gamma\}$.
In view of the above, this implies that $\mathbf P_n\{\mathbf w^\prime\approx\mathbf v^\prime\}$ can be given within $\mathbf P_n$ by some subset of $\Phi$, and we are done.

\smallskip

\textit{Distributivity of $\mathfrak L(\mathbf R_n)$}.
We are going to deduce the required fact from Lemma~\ref{L: smth imply distributivity} with $\mathbf V=\mathbf R_n$ and $\Sigma=\Phi$, where $\Phi$ consists of the identity $xy\approx yx$, all identities of the form $x^k \approx x^\ell$ with $k,\ell\in\mathbb N$, all identities of the form $\delta_{k,\ell}$ with $k,\ell\in\mathbb N_0$ and all identities of the form~\eqref{two letters in a block} with $r\in\mathbb N_0$, $e_0,f_0\in\mathbb N$, $e_1,f_1,\dots,e_r,f_r\in\mathbb N_0$, $\sum_{i=0}^r e_i\ge 2$ and $\sum_{i=0}^r f_i\ge 2$. 
In view of Lemmas~\ref{L: x^n=x^m in X wedge Y} and~\ref{L: xy.. = yx.. in X wedge Y} and Corollaries~\ref{C: xy=yx in X wedge Y} and~\ref{C: delta_{n,m} in X wedge Y}, to do this, it remains to prove only that each subvariety of $\mathbf R_n$ may be given within $\mathbf R_n$ by some subset of $\Phi$.

Let $\mathbf u \approx \mathbf v$ be an arbitrary identity.
It suffices to verify that $\mathbf R_n\{\mathbf u \approx \mathbf v\}=\mathbf R_n\Gamma$ for some $\Gamma\subseteq\Phi$.
If the variety $\mathbf R_n\{\mathbf u \approx \mathbf v\}$ is commutative, then it can be defined by a subset of $\{x^k\approx x^\ell,\,xy\approx yx\mid k,\ell\in\mathbb N\}$ by Proposition~\ref{P: commutative}.
So, it remains to consider the case when $\mathbf R_n\{\mathbf u \approx \mathbf v\}$ is not commutative.
In view of Proposition~4.1 in~\cite{Lee-12a} and the inclusion $\mathbf R_n\subseteq\var\{\alpha,\beta\}$, we may assume that one of the following two statements holds:
\begin{itemize}
\item[(a)] the identity $\mathbf u\approx\mathbf v$ coincides with an identity of the form~\eqref{one letter in a block} with $r,e_0,f_0,e_1,f_1,\dots,e_r,f_r\in\mathbb N_0$;{\sloppy

}
\item[(b)] the identity $\mathbf u\approx\mathbf v$ coincides with an identity of the form~\eqref{two letters in a block} with $r\in\mathbb N_0$, $e_0,f_0\in\mathbb N$, $e_1,f_1,\dots,e_r,f_r\in\mathbb N_0$, $\sum_{i=0}^r e_i\ge 2$ and $\sum_{i=0}^r f_i\ge 2$.
\end{itemize}

Notice that if the claim~(b) holds, then the identity $\mathbf u \approx \mathbf v$ lies in $\Phi$, and we are done.
Suppose now that the claim~(a) holds.
Lemma~\ref{L: one letter in block reduction} and the inclusion $\mathbf R_n\subseteq\mathbf A$ allow us to assume that $e_1,f_1,e_2,f_2,\dots,e_r,f_r\le 1$.
In view of Lemma~\ref{L: reduction to delta_{r,m}}, the identity~\eqref{one letter in a block} is equivalent within $\mathbf R_n$ to the set $\{x^e\approx x^f,\,\delta_{e-e_0,e_0},\,\delta_{f-f_0,f_0}\}$, where $e=\sum_{i=0}^r e_i$ and $f=\sum_{i=0}^r f_i$.
We see that the identity $\mathbf u \approx \mathbf v$ is equivalent within $\mathbf R_n$ to some subset of $\Phi$ in either case.

\smallskip

Theorem~\ref{T: A_cen} is proved.\qed

\begin{remark}
\label{R: countably infinite}
Analysing the proof of Theorem~\ref{T: A_cen}, one can notice that if $\mathbf X \in\{\mathbf P_n,\mathbf Q_n ,\mathbf R_n\}$, then each subvariety of $\mathbf X$ may be given within $\mathbf X$ by a finite set of identities.
Consequently, set of all distributive subvarieties of $\mathbf A_\mathsf{cen}$ is countably infinite.
\end{remark}

\subsection*{Acknowledgments.} The author thanks Edmond W.H. Lee, Boris M. Vernikov and Mikhail~V. Volkov for several comments and suggestions for improving the manuscript.{\sloppy

}

\end{document}